\documentclass{gtpart}
\usepackage{amsmath,amssymb,amsthm,stmaryrd}
\usepackage[all]{xy}
\usepackage[usenames,dvipsnames]{xcolor}
\usepackage{tikz}
\usetikzlibrary{shapes,positioning,fit}
\usepackage{url}
\usepackage{hyperref}
\usepackage{enumerate}
\usepackage{tensor}
\usepackage{mathrsfs}
\usepackage{graphicx}
\usepackage{mathtools}
\usetikzlibrary{arrows}

\title{The Hecke algebra action on Morava $E$-theory of height $2$}
\author{Yifei Zhu}
\givenname{Yifei}
\surname{Zhu}
\address{Department of Mathematics\\Northwestern University\\
         Evanston\\IL 60208\\USA}
\email{zyf@math.northwestern.edu}


\newtheorem{thm}[equation]{Theorem}

\newtheorem{prop}[equation]{Proposition}
\newtheorem{lem}[equation]{Lemma}
\theoremstyle{definition}
\newtheorem{defn}[equation]{Definition}
\newtheorem{cstr}[equation]{Construction}
\theoremstyle{remark}
\newtheorem{rmk}[equation]{Remark}
\newtheorem{ex}[equation]{Example}

\def\co{\colon\thinspace}
\newcommand{\mb}[1]{\mathbb{#1}}

\newcommand{\Aut}{{\rm Aut}}

\newcommand{\Proj}{{\rm Proj\thinspace}}
\newcommand{\Spf}{{\rm Spf\thinspace}}

\newcommand{\cF}{\overline {\mb F}}

\newcommand{\CC}{{\cal C}}
\newcommand{\CE}{{\cal E}}
\newcommand{\CF}{{\cal F}}
\newcommand{\CG}{{\cal G}}

\newcommand{\CM}{{\cal M}}
\newcommand{\CMB}{\overline {\cal M}}
\newcommand{\CO}{{\cal O}}
\newcommand{\CP}{{\cal P}}

\newcommand{\Frob}{{\rm Frob}}

\newcommand{\DL}{Dyer-Lashof~}

\newcommand{\BC}{{\mb C}}

\newcommand{\BF}{{\mb F}}
\newcommand{\BG}{{\mb G}}
\newcommand{\BN}{{\mb N}}
\newcommand{\BP}{{\mb P}}
\newcommand{\BQ}{{\mb Q}}

\newcommand{\BW}{{\mb W}}
\newcommand{\BZ}{{\mb Z}}

\newcommand{\HC}{\widehat{C~}\!}
\newcommand{\HCC}{\widehat{~\!\!\CC}}

\newcommand{\md}{~~{\rm mod}~}
\newcommand{\ad}{\text{and}}
\newcommand{\can}{{\rm can}}

\newcommand{\TMF}{{\rm TMF}}

\newcommand{\MF}{{\rm MF}}

\newcommand{\A}{\alpha}
\newcommand{\B}{\beta}
\renewcommand{\D}{\Delta}
\renewcommand{\d}{\delta}
\newcommand{\f}{\phi}
\newcommand{\G}{\Gamma}

\newcommand{\K}{\kappa}
\renewcommand{\l}{\lambda}
\renewcommand{\o}{\omega}
\newcommand{\ou}{\underline{\omega\!}}
\newcommand{\si}{\sigma}
\newcommand{\T}{\tau}

\newcommand{\ce}{\coloneqq}
\newcommand{\lb}{\llbracket}
\newcommand{\rb}{\rrbracket}
\newcommand{\lp}{(\!(}
\newcommand{\rp}{)\!)}

\newcommand{\wt}[1]{\textcolor{white}{#1} \!~}
\newcommand{\gl}{{\rm gl}}
\newcommand{\GL}{{\rm GL}}
\newcommand{\SL}{{\rm SL}}
\newcommand{\Tate}{{\rm Tate}}

\makeatletter
\DeclareRobustCommand\widecheck[1]{{\mathpalette\@widecheck{#1}}}
\def\@widecheck#1#2{%
    \setbox\z@\hbox{\m@th$#1#2$}%
    \setbox\tw@\hbox{\m@th$#1%
       \widehat{%
          \vrule\@width\z@\@height\ht\z@
          \vrule\@height\z@\@width\wd\z@}$}%
    \dp\tw@-\ht\z@
    \@tempdima\ht\z@ \advance\@tempdima2\ht\tw@ \divide\@tempdima\thr@@
    \setbox\tw@\hbox{%
       \raise\@tempdima\hbox{\scalebox{1}[-1]{\lower\@tempdima\box
\tw@}}}%
    {\ooalign{\box\tw@ \cr \box\z@}}}
\makeatother

\numberwithin{equation}{section}
\renewcommand{\theequation}{\thesection.\arabic{equation}}

\begin{document}

\begin{abstract}
 Given a one-dimensional formal group of height 2, let $E$ be the Morava 
 $E$-theory spectrum associated to its universal deformation over the Lubin-Tate 
 ring.  By computing with moduli spaces of elliptic curves, we give an 
 explicitation for an algebra of Hecke operators acting on $E$-cohomology.  This 
 leads to a vanishing result for Rezk's logarithmic cohomology operation on the 
 units of $E$.  It identifies a family of elements in the kernel with 
 meromorphic modular forms whose Serre derivative is zero.  Our calculation 
 finds a connection to logarithms of modular units.  In particular, we work out 
 an action of Hecke operators on certain ``logarithmic'' $q$-series, in the 
 sense of Knopp and Mason, that agrees with our vanishing result and extends the 
 classical Hecke action on modular forms.  
\end{abstract}

\maketitle

\section{Introduction}

In the context of elliptic cohomology, Hecke operators have been studied as 
cohomology operations by Baker \cite{Baker89,Baker90}, Ando \cite{Ando95}, and 
Ganter \cite{moonshine,orbifold,stringy}, among others.  The various cohomology 
theories each have as coefficients a ring of modular forms of a particular 
type---modular forms on $\SL_2(\BZ)$ with no condition at the cusp, $p$-adic 
modular forms, modular forms with level structure---or it is a closely related 
ring, e.g., a completion of the former at an ideal.  These cohomology theories 
are topological realizations of the domain that Hecke operators act on.  

The notion of an elliptic spectrum \cite[Definition 1.2]{AHS01} and the theory 
of multiplicative ring spectra, notably the theorem of Goerss, Hopkins, and 
Miller \cite[Corollary 7.6]{GH}, enable studying this action of Hecke operators 
via power operations.  These operations arise from the multiplicative ring 
structure on a spectrum, and they capture all the algebraic structure that 
naturally adheres to the cohomology theory represented by the spectrum 
\cite{H_infty,lpo}.  

In particular, Morava $E$-theories are a family of cohomology theories whose 
power operations are better understood thanks to the work of Ando, Hopkins, 
Strickland, and Rezk \cite{AHS04,cong,iph}.  This family is organized by heights 
$n$ and primes $p$.  Here the theory of formal groups plays a key role: each 
$E$-theory spectrum $E$ corresponds to a specific one-dimensional commutative 
formal group, whose finite flat subgroups, in turn, correspond precisely to 
power operations on $E$-cohomology.  At height $n = 2$, formal groups of 
elliptic curves provide concrete models for $E$-theories.  The associated power 
operations can then be computed from moduli spaces that parametrize isogenies 
between elliptic curves.  

In this paper, we study the action of Hecke operators on Morava $E$-theories at 
height 2, based on the explicit calculations of power operations in 
\cite{h2p2,p3}.  This is motivated by the interpretation---in terms of Hecke 
operators---for Rezk's logarithmic operations on $E$-cohomology at an arbitrary 
height (see \cite[1.12]{log}).  At height 2, these ``logarithms'' are critical 
in the work of Ando, Hopkins, and Rezk on rigidification of the string-bordism 
elliptic genus, or more precisely, on $E_\infty$-string orientations of the 
spectrum of topological modular forms \cite{koandtmf}.  For the space of such 
orientations, its set of components can be detected by elements in the kernel of 
a logarithm.  Explicitly, these elements are identified with Eisenstein series 
$\CE_k$, which are eigenforms of Hecke operators \cite[Theorem 12.3]{koandtmf} 
(cf.~\cite[Theorem 40]{Sprang}).  

In Section \ref{sec:kerlog}, we give a different account of certain elements in 
the kernel of a logarithm at height 2: they are meromorphic modular forms with 
vanishing Serre derivative, including modular forms whose zeros and poles are 
located only at the cusps, such as the modular discriminant $\D$.  This is 
stated as Theorem \ref{thm:kerlog}.  

The discrepancy between the above two sets of elements---Eisenstein series as 
opposed to the discriminant---results from the different domains of the 
logarithms.  The former is the group of units in the zeroth cohomology of an 
even-dimensional sphere, while the latter is with respect to the 
infinite-dimensional complex projective space $\BC\BP^\infty$, that is, the 
logarithms are defined on $E^0(S^{2 k})^\times$ and $E^0(\BC\BP^\infty)^\times$ 
respectively.  

The finiteness of $E^0(S^{2 k})$ as a module over $E^0$ leads to a simple 
formula for the logarithm.  Specifically, the logarithm can be written as a 
combination of Hecke operators acting on $\log (1 + f)$, where $f$ is a 
generator of the truncated polynomial ring $E^0(S^{2 k})$.  Note that the formal 
power series expansion of $\log (1 + f)$ simply equals $f$, because $f^2 = 0$ 
(see \cite[Proposition 4.8 and Example 4.9]{koandtmf}).  

In the latter case of $E^0(\BC\BP^\infty)$, we calculate instead with $\log(g)$ 
for units $g$ in $E^0$.  Through a generator $u$ of the formal power series ring 
$E^0(\BC\BP^\infty)$, certain $g$ can be represented by meromorphic modular 
forms (cf.~Proposition \ref{prop:mfe0}).  Without nilpotency of $g$ as in the 
case studied by Ando, Hopkins, and Rezk, a different set of tools from number 
theory is applied.  In particular, it came as a surprise to the author that the 
formula of Rezk for the logarithms, which arose from a purely homotopy-theoretic 
construction, resembled the {\em logarithms of ratios of Siegel functions} 
studied by Katz in \cite[Section 10.1]{padicinterp} (see Remark 
\ref{rmk:ratio}).  Indeed, Katz's approach to those logarithms inspired a final 
step in our proof of Theorem \ref{thm:kerlog}.  

Section \ref{sec:logq} provides a ``feedback'' to number theory.  As mentioned 
above, various types of modular forms have been associated with elliptic 
cohomology theories.  In our calculation of the logarithms on $E$-cohomology, 
the occurrence of $\log(g)$ indicates that Morava $E$-theories at height 2 
witness a larger class of functions on elliptic curves than modular forms.  In 
fact, for meromorphic modular forms $g$ in a $q$-expansion, $\log(g)$ are {\em 
logarithmic} $q$-series studied by Knopp and Mason \cite{KnoppMason}.  A 
paradigm is $\log(\D)$, which we discuss in Example \ref{ex:logD} and Remark 
\ref{rmk:eichler}.  Motivated by the logarithmic operations from homotopy 
theory, we sketch how the domain of Hecke operators may extend to include 
certain logarithmic $q$-series.  This is Definition \ref{def:logq}, from which 
the statement for the case $f = \D$ in Theorem \ref{thm:kerlog} falls out.  

This interplay via Hecke operators between homotopy theory and number theory is 
bookended by the foundational material in Sections \ref{sec:mf2E}--\ref{sec:ct} 
and further discussions in Section \ref{sec:individual} on Hecke operators as 
additive power operations.  

Crucial to our understanding of the action of Hecke operators are the explicit 
computations for power operations.  Building on the prime-2 case in \cite{h2p2} 
and the prime-3 case in \cite{p3}, we present in Section \ref{subsec:po} a 
general recipe for computing power operations on Morava $E$-theories at height 2 
for all primes, and make the prime-5 calculations available as a working example 
throughout the paper (Examples \ref{ex:mfe0}, \ref{ex:po}, \ref{ex:ho}, 
\ref{ex:log}, \ref{ex:gamma}, and \ref{ex:t5}).  The explicit formulas show that 
the ``topological'' Hecke operators---constructed directly from power operations 
as in \cite[Proposition 3.6.2]{Ando95}---do not agree with the classical Hecke 
operators acting on modular forms (Remark \ref{rmk:tc}).  This explicitation 
also shows that a topological Hecke operator may not commute with all additive 
power operations on $E$-cohomology (Theorem \ref{thm:center}).  
\vspace{-.07in}

\subsection{Acknowledgements}

I thank Andrew Baker, Paul Goerss, Tyler Lawson, Charles Rezk, and Joel Specter 
for helpful discussions.  

I thank my friends Tzu-Yu Liu and Meng Yu, and their parents, for the 
hospitality I received during me writing this paper in their home in California.

\subsection{Conventions}

Throughout the paper, $p$ will denote a prime, and $N$ an integer prime to $p$ 
with $N > 3$.  

Let $\MF\big(\G_1(N)\big)$ be the graded ring of modular forms of level 
$\G_1(N)$ over $\BZ[1/N]$ with no condition at the cusps 
(cf.~\cite[Section 1.2]{padicprop}).  Complex-analytically, this corresponds to 
the graded complex vector space of weakly holomorphic modular forms on $\G_1(N)$ 
(cf.~\cite[Definition 1.12]{web}).  For simplicity, we will refer to these as 
{\em modular forms of level $\G_1(N)$}, or {\em modular forms} if the congruence 
subgroup $\G_1(N)$ is clear from the context.  

By a {\em meromorphic modular form}, we mean a complex-valued modular form that 
is meromorphic both at the cusps and over the upper half of the complex plane 
\cite[Definition 1.8]{web}.  

We use a list of symbols.  
\begin{center}
 \begin{tabular}{ll}
  $\BF_p$ & a field with $p$ elements \\
  $\BZ_p$ & the ring of $p$-adic integers \\
  $\BQ_p$ & the field of $p$-adic numbers \\
  $\overline{k}$ & a separable closure of a field $k$ \\
  $\BW(k)$ & the ring of $p$-typical Witt vectors over a field $k$ of 
             characteristic $p$ \\
  $\Sigma_m$ & the symmetric group on $m$ letters \\
  $R^\times$ & the multiplicative group of invertible elements in a ring $R$ \\
  $R_I^\wedge$ & the completion of a ring $R$ with respect to an ideal $I$ \\
  $R(\!(x)\!)$ & the ring of formal Laurent series over a ring $R$ in a variable 
                 $x$ \\
               & write $\sum_{n > -\infty} a_n x^n$ for a general element in 
                 this ring \\
  $C/R$ \, or \, $C_R$ & a scheme $C$ over a ring $R$ \\
  $[m]$ & the multiplication-by-$m$ map on a group scheme \\
  $C[m]$ & the kernel of $[m]$ on a group scheme $C$ \\
  $\HC$ & the formal completion of an elliptic curve $C$ at the identity \\
  $q$ & $e^{2 \pi i z}$ for any $z \in \BC$ \\
  $\si_s(m)$ & $\displaystyle \sum_{\stackrel{\scriptstyle d~\!|~\!m}{d \geq 1}} 
               d^s$ for any positive integer $m$, with $s \in \BZ$ 
 \end{tabular}
\end{center}

\subsection{Correction to two cited formulas}
\label{subsec:correction}

For the reader's convenience, we list the following typographical corrections.  
\begin{itemize}
 \item \cite[(1.11.0.4)]{padicprop} The last line should start with $\ell^{-k}$ 
 instead of $\ell^{-1}$.  

 \item \cite[Problem 16-A]{cc} The left-hand side of Girard's formula should 
 read $(-1)^n s_n / n$.  The summation on the right-hand side is over 
 $i_1 + 2\,i_2 + \cdots + k\,i_k = n$.  
\end{itemize}

\section{From rings of modular forms to homotopy groups of\\
         Morava $E$-theories}
\label{sec:mf2E}

In this section, we explain how every modular form of level $\G_1(N)$ gives an 
element in the zeroth coefficient ring of a Morava $E$-theory at height 2 and 
prime $p$.  The connection comes from finding explicit models for the 
$E$-theory.

\subsection{Models for an $E$-theory of height 2}
\label{subsec:model}

Morava $E$-theory spectra can be viewed as topological realizations of 
Lubin-Tate rings.  Specifically, given a formal group $\BG_0$ of height 
$n < \infty$ over a perfect field $k$ of characteristic $p$, the associated 
Morava $E$-theory (of height $n$ at the prime $p$) is a complex-oriented 
cohomology theory $E$ whose formal group $\Spf E^0 \BC\BP^\infty$ is the 
universal deformation of $\BG_0$ over the Lubin-Tate ring 
\[
 \BW \big( \cF_p \big) \lb u_1, \ldots, u_{n - 1} \rb \cong E^0 
\]
(see \cite[Section 3]{LT} and \cite[Section 7]{GH}).  For height $n = 2$, via 
the Serre-Tate theorem \cite{LST} (cf.~\cite[Theorem 2.9.1]{KM}), this universal 
deformation of a formal group can be obtained from a universal deformation of a 
supersingular elliptic curve.  We construct the latter as the universal family 
of elliptic curves equipped with a level-$\G_1(N)$ structure.  

Write $\CP_N$ for the representable moduli problem of smooth elliptic curves 
over $\BZ[1/N]$ with a choice of a point $P_0$ of exact order $N$ and a 
nonvanishing one-form $\o$.  The following two universal families represent 
$\CP_N$ for $N = 4$ and $N = 5$ respectively, and they suffice to model 
$E$-theories of height 2 at all primes.  

\begin{ex}[{\cite[Proposition 2.1]{p3}}]
 \label{ex:4}
 The moduli problem $\CP_4$ is represented by the curve 
 \[
  \CC_4 \co y^2 + A x y + A B y = x^3 + B x^2 
 \]
 over the graded ring 
 \[
  S_4 \ce \BZ[1/4][A, B, \D^{-1}] 
 \]
 where $|A| = 1$, $|B| = 2$, and $\D = A^2 B^4 (A^2 - 16 B)$.  The chosen point 
 is $P_0 = (0,0)$ and the chosen one-form is $\o = du$ with $u = x / y$.  
\end{ex}

\begin{ex}[{\cite[Corollary 1.1.10]{tmf5}}]
 \label{ex:5}
 The moduli problem $\CP_5$ is represented by 
 \[
  \CC_5 \co y^2 + A x y + B^2 (A - B) y = x^3 + B (A - B) x^2 
 \]
 over the graded ring 
 \[
  S_5 \ce \BZ[1/5][A, B, \D^{-1}] 
 \]
 where $|A| = |B| = 1$ and $\D = B^5 (B - A)^5 (A^2 + 9 A B - 11 B^2)$.  Again 
 the chosen point is $P_0 = (0,0)$ and the chosen one-form is $\o = du$ with 
 $u = x / y$.  Moreover, writing $\zeta \ce e^{2 \pi i / 5}$, we have $\D = B^5 
 (B - A)^5 \big( A - (5 \zeta^4 + 5 \zeta - 2) B \big) \big( A + (5 \zeta^4 + 5 
 \zeta + 7) B \big)$.  
\end{ex}

\begin{rmk}
 \label{rmk:tmf3}
 The moduli problem $\CP_3$ is also representable \cite[Proposition 3.2]{tmf3}, 
 but $[\G_1(3)]$ is not (cf.~\cite[Corollary 2.7.4]{KM}).  The beginning of 
 \cite[Section 1]{tmf5} explains the relationship between these two moduli 
 problems.  For the prime 2, compare \cite[Corollary 1.1.11]{tmf5} and 
 \cite[Section 3.1]{level3II} (see also Proposition \ref{prop:tmfe} below).  
\end{rmk}

From each of the above examples, restricting $\CC_N$ over a closed point in the 
supersingular locus at $p$, we get a supersingular elliptic curve $C_0$ over 
$\cF_p$.  By the Serre-Tate theorem \cite{LST} (cf.~\cite[Theorem 2.9.1]{KM}), 
the formal completion $\HCC_N$ of $\CC_N$ at the identity then gives the 
universal deformation of the formal group $\HC_0$.  Here $\HC_0 / \cF_p$ is of 
height 2, and it models $\BG_0 / k$ for the $E$-theory we begin with.  

\begin{rmk}
 \label{rmk:N}
 For a fixed $E$-theory, the various models each involve a choice of $N$ for the 
 $\CP_N$-structures, and a choice of an isomorphism class of supersingular 
 elliptic curves over $\cF_p$ equipped with a $\CP_N$-structure.  Over a 
 separably closed field of characteristic $p$, any two formal group laws of the 
 same height are isomorphic \cite[Th\'eor\`eme IV]{Lazard}.  In view of this and 
 the universal property in the Lubin-Tate theorem \cite[Theorem 3.1]{LT}, we see 
 that up to isomorphism the $E$-theory is independent of these choices.  
\end{rmk}

Topologically, there is a $K(2)$-localization that corresponds to the above 
completion along the supersingular locus at $p$, as we now describe.  

The graded rings $S_N$ in Examples \ref{ex:4} and \ref{ex:5} can be identified 
as $\MF\big(\G_1(N)\big)$.  Their topological realizations are the periodic 
spectra $\TMF\big(\G_1(N)\big)$ of topological modular forms of level $\G_1(N)$ 
(see \cite[Section 2]{tmf3} and cf.~\cite{logetaletmf}).  Following the 
convention that elements in algebraic degree $k$ lie in topological degree 
$2 k$, we have 
\[
 \pi_* \TMF\big(\G_1(4)\big) \cong \BZ[1/4][A, B, \D^{-1}] 
\]
with $|A| = 2$, $|B| = 4$, and $\D = A^2 B^4 (A^2 - 16 B)$, and 
\[
 \pi_* \TMF\big(\G_1(5)\big) \cong \BZ[1/5][A, B, \D^{-1}] 
\]
with $|A| = |B| = 2$ and $\D = B^5 (B - A)^5 (A^2 + 9 A B - 11 B^2)$.  

\begin{prop}[{cf.~\cite[Section 3.5]{BOSS} and 
\cite[Corollary on page 20]{Tnf}}]
 \label{prop:tmfe}
 Let $K(2)$ be the Morava $K$-theory spectrum at height $2$ and prime $p$, with 
 $\pi_* K(2) \cong \cF_p [v_2^{\pm 1}]$, where $|v_2| = 2 (p^2 - 1)$.  Let 
 $N > 3$ be an integer prime to $p$.  Denote by $L_{K(2)} \TMF\big(\G_1(N)\big)$ 
 the Bousfield localization of $\TMF\big(\G_1(N)\big)$ with respect to $K(2)$.  
 Given a Morava $E$-theory spectrum $E$ of height $2$ at the prime $p$, there is 
 a noncanonical isomorphism 
 \[
  L_{K(2)} \TMF\big(\G_1(N)\big) \cong 
  \underbrace{E \times \cdots \times E}_\text{$m$ copies} 
 \]
 of $E_\infty$-ring spectra, where $m$ is the number of isomorphism classes of 
 supersingular elliptic curves over $\cF_p$ equipped with a level-$\G_1(N)$ 
 structure.  
\end{prop}

\begin{proof}
 Let $C_0$ be a supersingular elliptic curve over $\cF_p$ equipped with a level 
 $\G_1(N)$-structure.  Its formal group $\HC_0 / \cF_p$ gives a model for the 
 $E$-theory.  By the Goerss-Hopkins-Miller theorem \cite[Corollary 7.6]{GH}, the 
 spectrum $E$ admits an action of the automorphism group $\Aut(C_0 / \cF_p)$.  

 Let $G$ be the subgroup of $\Aut(C_0 / \cF_p)$ consisting of automorphisms that 
 preserve the $\G_1(N)$-structure on $C_0$.  In view of Remark \ref{rmk:N}, we 
 then obtain $L_{K(2)} \TMF\big(\G_1(N)\big)$ as taking homotopy fixed points 
 for the action of $G$ on $E$, one such copy for each closed point in the 
 supersingular locus at $p$.  By \cite[Corollary 2.7.4]{KM}, the moduli problem 
 $[\G_1(N)]$ is rigid when $N > 3$, and thus $G$ is trivial.  Again by the 
 Goerss-Hopkins-Miller theorem, we get the stated isomorphism as one between 
 $E_\infty$-ring spectra.  
\end{proof}

\subsection{Homotopy groups of an $E$-theory at height 2}
\label{subsec:mfe0}

Thanks to the explicit models for $E$-theories, the global-to-local relationship 
between $\TMF\big(\G_1(N)\big)$ and $E$ in Proposition \ref{prop:tmfe} can be 
spelled out on homotopy groups.  The next example illustrates the passage from 
$\pi_* \TMF\big(\G_1(N)\big)$ to 
\[
 E_* \cong \BW \big( \cF_p \big) \lb u_1 \rb [u^{\pm 1}] 
\]
where the ``deformation'' parameter $u_1$ in degree 0 comes from a Hasse 
invariant, and the 2-periodic unit $u$ in degree $-2$ corresponds to a local 
uniformizer at the identity of the universal elliptic curve $\CC_N$.  

\begin{ex}
 \label{ex:mfe0}
 Let $p = 5$ and $N = 4$.  

 By \cite[V.4.1a]{AEC}, the Hasse invariant of $\CC_4$ at the prime 5 equals 
 $A^4 - A^2 B + B^2 \in \BF_5[A, B]$.  For a reason that will become clear in 
 \eqref{W}, we choose an integral lift of this Hasse invariant given by 
 \[
  H \ce A^4 - 16 A^2 B + 26 B^2 \in S_4 
 \]
 At the prime 5, the supersingular locus of $\CC_4$ is then the closed subscheme 
 of $\Proj S_4$ cut out by the ideal $(5, H)$.  It consists of a single closed 
 point, as $H$ is irreducible over $\BF_5$.  Since $\D = A^2 B^4 (A^2 - 16 B)$ 
 gets inverted in $S_4$, the scheme $\Proj S_4$ is affine, and is contained in 
 the affine open chart 
 \[
  \Proj \BZ[1/4] [A, B] [A^{-1}] 
 \]
 for the weighted projective space $\Proj \BZ[1/4] [A, B]$.  There is another 
 affine chart 
 \[
  \Proj \BZ[1/4] [A, B] [u] 
 \]
 with $u^2 = B^{-1}$ that is \'etale over $\Proj \BZ[1/4] [A, B]$.  The 
 supersingular point is contained in both charts, as illustrated in the 
 following diagram.  
 \\
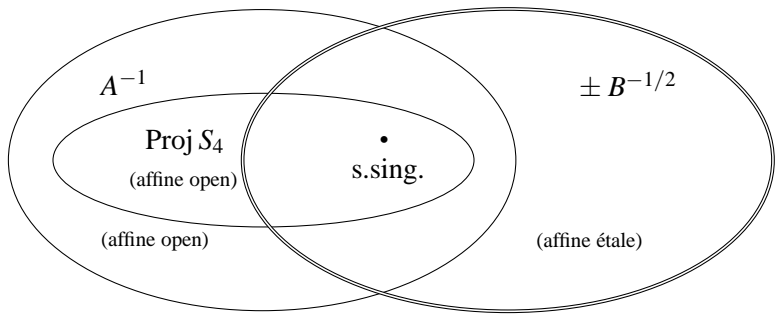
 \begin{equation*}
  \begin{tikzpicture}[baseline=(current bounding box.center), 
  every fit/.style={ellipse,draw,inner sep=0pt}, mytext/.style={inner sep=6pt}]
   \node[mytext,align=left,rectangle] (a) at (0, 0) 
        {$A^{-1}$ \\ \\ \\ \\ $\scriptstyle \text{(affine open)}$}; 
   \node[mytext,align=left,rectangle] (b) at (6, 0) 
        {$\quad~ \pm B^{-1/2}$ \\ \\ \\ \\ $\scriptstyle 
         \text{(affine \'etale)}$}; 
   \node[mytext,align=left,rectangle] (c) at (3, 0) 
        {$\quad~\, _\bullet$ \\ ~ s.sing.}; 
   \node[mytext,align=left,rectangle] (c') at (3.2, 0) {}; 
   \node[mytext,align=left,rectangle] (d) at (.4, 0) 
        {$~\, \Proj S_4$ \\ $\scriptstyle \!\text{(affine open)}$}; 
   \node[draw,fit=(a) (c)] {}; 
   \node[draw,fit=(d) (c')] {}; 
   \node[draw,double,fit=(b) (c)] {}; 
  \end{tikzpicture}
 \end{equation*}
 We now pass to the homotopy groups of the corresponding $E$-theory spectrum 
 $E$.  Define elements 
 \[
  a \ce u A, \quad h \ce u^4 H = a^4 - 16 a^2 + 26, \quad \ad \quad \d \ce 
  u^{12} \D = h - 26 
 \]
 Following the convention that elements in algebraic degree $k$ lie in 
 topological degree $2 k$, we then have 
 \[
  E_* \cong \BW \big( \cF_5 \big) \lb h \rb [u^{\pm 1}] 
 \]
 with $|h| = 0$ and $|u| = -2$.  In particular, by Hensel's lemma, both $a$ and 
 $\d$ are contained in $\big( E_0 \big)^{\!\times}$.  Moreover, $u$ corresponds 
 to a coordinate on $\HCC_4$ via a chosen isomorphism 
 $\Spf E^0 \BC\BP^\infty \cong \HCC_4$ of formal groups over $\Spf E^0 \cong 
 \Spf\!\left( {S_4}^\wedge_{(5,H)} \otimes_{\BW(\BF_5)} 
 \BW \big( \cF_5 \big) \!\right)$ (cf.~\cite[Definition 1.2]{AHS01}).  
\end{ex}

\begin{rmk}
 \label{rmk:abuse}
 As an abuse of notation, in Examples \ref{ex:4} and \ref{ex:5} we have written 
 \[
  u = \frac{~\!x~\!}{y} 
 \]
 where $x$ and $y$ are the affine coordinates in the Weierstrass equation for 
 $\CC_N$.  The resulting algebraic degree $-1$ of $u$ matches the topological 
 degree $-2$ of $u$ in $E_*$.  
\end{rmk}

The process in the previous example applies to all pairs $p$ and $N$.  For 
certain pairs, the supersingular locus contains more than one closed point, 
e.g., the Hasse invariant of $\CC_4$ at $p = 11$ factors as 
$(A^2 + B) (A^8 + 3 A^6 B + 4 A^2 B^3 + B^4)$.  In this case, $H$ lifts one of 
the irreducible factors of the Hasse invariant.  The corresponding closed point 
carries the supersingular elliptic curve whose formal group models $\BG_0 / k$ 
for the $E$-theory.  

Recall that elements in $\MF\big(\G_1(N)\big)$ are functions $f$ on elliptic 
curves $C/R$ equipped with a $\CP_N$-structure $(P_0,\omega)$.  Each value 
$f(C/R, P_0, \omega) \in R$ depends only on the $R$-isomorphism class of the 
triple $(C/R, P_0, \omega)$, and is subject to a modular transformation property 
that encodes the weight of $f$; moreover, its formation commutes with arbitrary 
base change (see, e.g., \cite[Section 1.2]{padicprop}).  

\begin{prop}[cf.~{\cite[Lemma 6.3]{BOSS}}]
 \label{prop:mfe0}
 Let $E$ be a Morava $E$-theory of height $2$ at the prime $p$.  There is a 
 composite 
 \[
  \B_{p,\,N} \co \MF\big(\G_1(N)\big) \hookrightarrow E^0 \times \BZ \to E^0 
 \]
 of ring homomorphisms, where the first map is injective and the second map is 
 the projection.  
\end{prop}

\begin{proof}
 In view of Remark \ref{rmk:N}, given a model for the $E$-theory as in Example 
 \ref{ex:mfe0}, a modular form $f$ of weight $k$ maps to 
 $\big( u^k \cdot \, f(\CC_N, P_0, du), ~ k \big) \in E^0 \times \BZ$.  The 
 injectivity of this map follows from the universal property of the triple 
 $(\CC_N, P_0, du)$.  
\end{proof}

We will drop the subscripts in ``$\B_{p,\,N}$'' when there is no ambiguity.

\section{From classical to topological Hecke operators}
\label{sec:ct}

Morava $E$-theory spectra are $E_\infty$-ring spectra \cite[Corollary 7.6]{GH}, 
and thus they come equipped with power operations.  The work of Ando, Hopkins, 
and Strickland \cite{AHS04} (cf.~\cite[Theorem B]{cong}) sets up a 
correspondence between power operations on an $E$-theory and deformations of 
Frobenius isogenies of formal groups.  

At height 2, the Serre-Tate theorem \cite{LST} (cf.~\cite[Theorem 2.9.1]{KM}) 
sets up a second correspondence between isogenies of formal groups and isogenies 
of elliptic curves, in terms of their deformation theory.  

Via these two bridges, the classical action of Hecke operators on modular forms 
then corresponds to an action of ``topological'' Hecke operators on 
$E$-cohomology \cite[Section 14]{log}.  The purpose of this section is to give 
an explicit comparison between these two actions.  In particular, we will 
explain why the ring homomorphism $\B$ in Proposition \ref{prop:mfe0} 
{\em fails} to become a map of modules over a Hecke algebra.

\subsection{Isogenies and power operations}
\label{subsec:po}

Let $\CC_N$ over $S_N$ be the universal curve in Section \ref{subsec:model} for 
the moduli problem $\CP_N$.  Denote by $\CG_N^{(p)}$ the universal example of a 
degree-$p$ subgroup of $\CC_N$.  It is defined over an extension ring 
$S_N^{(p)}$ that is free of rank $p + 1$ as an $S_N$-module 
\cite[Theorem 6.6.1]{KM}.  Explicitly, 
\[
 S_N^{(p)} \cong S_N[\K] / \big(W(\K)\big) 
\]
where $\K$ is a generator that satisfies $W(\K) = 0$ for a monic polynomial $W$ 
of degree $p + 1$.  The roots $\K_0$, $\K_1$, \ldots, $\K_p$ of $W$ each 
correspond to a degree-$p$ subgroup of $\CC_N$.  In particular, let $\K_0$ 
correspond to the subgroup whose restriction over an ordinary point is the 
unique degree-$p$ subgroup of $\HCC_N$.  

We write 
\[
 \Psi_N^{(p)} \co \CC_N \to \CC_N / \CG_N^{(p)} 
\]
for the universal degree-$p$ isogeny over $S_N^{(p)}$, and write $\CC_N^{(p)}$ 
for the quotient curve $\CC_N / \CG_N^{(p)}$.  Following Lubin 
\cite[proof of Theorem 1.4]{Lubin67}, we construct $\Psi_N^{(p)}$ as a 
deformation of Frobenius, that is, over any closed point in the supersingular 
locus at $p$, $\Psi_N^{(p)}$ restricts as the $p$-power Frobenius endomorphism 
on the corresponding supersingular elliptic curve.  

\begin{cstr}[{cf.~\cite[proof of Theorem 1.4]{Lubin67} and 
\cite[Section 7.7]{KM}}]
 \label{cstr}
 Let $P$ be any point on $\CC_N$.  Write $u = x / y$ and $v = 1 / y$, where $x$ 
 and $y$ are the usual coordinates in an affine Weierstrass equation with the 
 identity $O$ at the infinity.  We have $\big( u(O), v(O) \big) = (0,0)$, and 
 $u$ is a local uniformizer at $O$.  
 \begin{enumerate}[(i)]
  \item  Define $\Psi_N^{(p)} \co \CC_N \to \CC_N^{(p)}$ by the formula 
  \[
   u\big(\Psi_N^{(p)}(P)\big) \ce \prod_{Q \in \CG_N^{(p)}} u(P - Q) 
  \]

  \item \label{K}  Define $\K \in S_N^{(p)}$ in degree $-p + 1$ as 
  \[
   \K \ce \prod_{Q \, \in \, \CG_N^{(p)} \setminus \! \{O\}} u(Q) 
  \]
 \end{enumerate}
\end{cstr}

We verify that the isogeny $\Psi_N^{(p)}$ has kernel precisely the subgroup 
$\CG_N^{(p)}$.  Moreover, it is a deformation of Frobenius, since at a 
supersingular point the $p$-divisible group is formal so that $Q = O$ for all 
$Q \in \CG_N^{(p)}$.  

\begin{rmk}
 \label{rmk:K}
 The element $\K$ in Construction \ref{cstr}\,\eqref{K} is the ``norm'' 
 parameter for the moduli problem $[\G_0(p)]$ as an ``open arithmetic surface'' 
 (see \cite[Section 7.7]{KM}).  As $u$ is a local coordinate near the identity 
 $O$, we note that by the above construction the cotangent map 
 $\big(\Psi_N^{(p)}\big)^*$ at $O$ sends $du$ to $\K du$.  
\end{rmk}

Via completion at a supersingular point, we see in Section \ref{subsec:mfe0} 
that the ring $S_N \cong \MF\big(\G_1(N)\big)$ representing $\CP_N$ is locally 
realized in homotopy theory as $E^0$.  Here, given the ring $S_N^{(p)}$ 
representing the simultaneous moduli problem $\big(\CP_N,[\G_0(p)]\big)$, 
Strickland's theorem \cite[Theorem 1.1]{Str98} identifies the completion of 
$S_N^{(p)}$ at the supersingular point as $E^0(B\Sigma_p) / I$, where $I$ is a 
transfer ideal.  In particular, $E^0(B\Sigma_p) / I$ is free over $E^0$ of rank 
$p + 1$, isomorphic to $E^0[\A] / \big( w(\A) \big)$ for a monic polynomial $w$ 
of degree $p + 1$.  

The universal degree-$p$ isogeny $\Psi_N^{(p)} \co \CC_N \to \CC_N^{(p)}$ is 
constructed above as a deformation of Frobenius.  By \cite[Theorem B]{cong} it 
then corresponds to a {\em total power operation} 
\begin{equation}
 \label{psi}
 \psi^p \co A^0 \to A^0(B\Sigma_p) / J 
\end{equation}
where $A$ is any $K(2)$-local commutative $E$-algebra, and $J$ is the 
corresponding transfer ideal.  Since $E^0(B\Sigma_p)$ is free over $E^0$ of 
finite rank \cite[Theorem 3.2]{Str98}, we have $J \cong A^0 \otimes_{E^0} I$ and 
\[
 A^0(B\Sigma_p) / J \cong \big( A^0 \otimes_{E^0} E^0(B\Sigma_p) \big) / J \cong 
 A^0 \otimes_{E^0} \big( E^0(B\Sigma_p) / I \big) \cong 
 A^0[\A] / \big( w(\A) \big) 
\]
Note that up to isomorphism, $\psi^p$ is independent of the choice of $N$ (see 
Remark \ref{rmk:N}).  Moreover, taking quotient by the transfer ideal makes this 
operation additive and hence a local ring homomorphism.  

\begin{ex}
 \label{ex:po}
 We continue Example \ref{ex:mfe0} with $p = 5$ and $N = 4$.  

 The universal degree-$5$ isogeny $\Psi_4^{(5)} \co \CC_4 \to \CC_4^{(5)}$ is 
 defined over the graded ring $S_4^{(5)} \cong S_4[\K] / \big(W(\K)\big)$, where 
 $|\K| = -4$ and 
 \begin{equation}
  \label{W}
  W(\K) = \K^6 - \frac{10}{B^2} \K^5 + \frac{35}{B^4} \K^4 - \frac{60}{B^6} \K^3 
  + \frac{55}{B^8} \K^2 - \frac{H}{B^{12}} \K + \frac{5}{B^{12}} 
 \end{equation}
 This polynomial is computed from the division polynomial $\psi_5$ in 
 \cite[Exercise 3.7]{AEC} for the curve $\CC_4$.  We first deduce from $\psi_5$ 
 identities satisfied by the $uv$-coordinates of a universal example 
 $Q \in \CG_4^{(5)} \!\setminus\!\! \{O\}$ as in 
 \cite[proof of Proposition 2.2]{p3}.  We then compute an explicit formula for 
 \[
  \K = u(Q) \cdot u(-Q) \cdot u(2 Q) \cdot u(-2 Q) 
 \]
 using methods analogous to \cite[III.2.3]{AEC}.  Finally we solve for a monic 
 degree-6 equation satisfied by $\K$ and obtain the polynomial $W$.  (We do not 
 compute an equation for $\CC_4^{(5)}$ as in \cite[Proposition 2.3]{p3}.)  

 Passing to the corresponding power operation, we write 
 \begin{equation}
  \label{A}
  \A \ce u^{-4} \K_0 
 \end{equation}
 (see Remark \ref{rmk:abuse}) where $\K_0$ corresponds to the subgroup of 
 $\CC_4$ whose restriction over an ordinary point is the unique degree-$5$ 
 subgroup of $\HCC_4$.  The total power operation 
 $\psi^5 \co E^0 \to E^0(B\Sigma_5) / I$ then lands in 
 \[
  E^0(B\Sigma_5) / I \cong \BW \big( \cF_5 \big) \lb h, \A \rb / \big(w(\A)\big) 
 \]
 where $w(\A) = \A^6 - 10 \A^5 + 35 \A^4 - 60 \A^3 + 55 \A^2 - h \A + 5$.  To 
 compute the effect of $\psi^5$ on the generator $h \in E^0$, we proceed as 
 follows.  

 Consider a second universal degree-$5$ isogeny 
 \[
  \quad \widetilde{\Psi}_4^{(5)} \co \CC_4^{(5)} \to 
  \CC_4^{(5)} / \widetilde{\CG}_4^{(5)} 
 \]
 where $\widetilde{\CG}_4^{(5)} = \CC_4[5] / \CG_4^{(5)}$.  It is defined the 
 same way as in Construction \ref{cstr}, with a parameter 
 $\widetilde{\K} \in S_4^{(5)}$ in degree $-20$.  Over $S_4^{(5)}$, the 
 assignment 
 \[
  \qquad\qquad~ \left(\! \big( \CC_4, P_0, du \big), ~ \CG_4^{(5)} \right) 
  \mapsto \left(\! \big( \CC_4^{(5)}, \Psi_4^{(5)}(P_0), du \big), 
  ~ \widetilde{\CG}_4^{(5)} \right) 
 \]
 is an involution on the moduli problem $\big(\CP_4,[\G_0(5)]\big)$ 
 (cf.~\cite[11.3.1]{KM}).  In particular, by rigidity, we have the identity 
 \[
  \widetilde{\Psi}_4^{(5)} \circ \Psi_4^{(5)} = [5] \quad~~ 
 \]
 that lifts $\Frob^2 = [5]$ over the supersingular point, where $\Frob$ is the 
 5-power Frobenius endomorphism.  Thus in view of Remark \ref{rmk:K} we obtain a 
 relation 
 \[
  \quad~ \widetilde{\K} \cdot \K = \frac{5}{B^{12}} 
 \]
 in $S_4^{(5)}$ (cf.~\cite[proof of Corollary 3.2]{p3}).  

 Correspondingly, there is an involution 
 \[
  \quad~~ (h, \A) \, \mapsto \, (\widetilde{h}, \widetilde{\A}) 
 \]
 on $E^0(B\Sigma_5) / I$, coming from the Atkin-Lehner involution of modular 
 forms on $\G_0(5)$ (cf.~\cite[Lemmas 7--10]{AtkinLehner}).  In particular, the 
 relation 
 \begin{equation}
  \label{w}
  \A^6 - 10 \A^5 + 35 \A^4 - 60 \A^3 + 55 \A^2 - h \A + 5 = 0 
 \end{equation}
 has an analog 
 \begin{equation}
  \label{w'}
  \widetilde{\A}^6 - 10 \widetilde{\A}^5 + 35 \widetilde{\A}^4 
  - 60 \widetilde{\A}^3 + 55 \widetilde{\A}^2 - \widetilde{h} \widetilde{\A} + 5 
  = 0 
 \end{equation}
 and we have 
 \begin{equation}
  \label{AA'}
  \widetilde{\A} \cdot \A = 5 
 \end{equation}
 Based on these, we compute that 
 \begin{equation}
  \label{psi5}
  \begin{split}
   \psi^5(h) = & ~ \widetilde{h} \\
             = & ~ \widetilde{\A}^5 - 10 \widetilde{\A}^4 + 35 \widetilde{\A}^3 
                 - 60 \widetilde{\A}^2 + 55 \widetilde{\A} + \A 
                 \qquad\qquad \text{by \eqref{w'} and \eqref{AA'}} \\
             = & ~ (-\A^5 + 10 \A^4 - 35 \A^3 + 60 \A^2 - 55 \A + h)^5 \\
               & - 10 (-\A^5 + 10 \A^4 - 35 \A^3 + 60 \A^2 - 55 \A + h)^4 \\
               & + 35 (-\A^5 + 10 \A^4 - 35 \A^3 + 60 \A^2 - 55 \A + h)^3 \\
               & - 60 (-\A^5 + 10 \A^4 - 35 \A^3 + 60 \A^2 - 55 \A + h)^2 \\
               & + 55 (-\A^5 + 10 \A^4 - 35 \A^3 + 60 \A^2 - 55 \A + h) \\
               & + \A 
                 \qquad\qquad\qquad\qquad\qquad\qquad\qquad\qquad\qquad 
                 \text{by \eqref{AA'} and \eqref{w}} \\
             = & ~ h^5 - 10 h^4 - 1065 h^3 + 12690 h^2 + 168930 h \\
               & - 1462250 + (-55 h^4 + 850 h^3 + 39575 h^2 \\
               & - 608700 h - 1113524) \A + (60 h^4 - 775 h^3 \\
               & - 45400 h^2 + 593900 h + 2008800) \A^2 + (-35 h^4 \\
               & + 400 h^3 + 27125 h^2 - 320900 h - 1418300) \A^3 \\
               & + (10 h^4 - 105 h^3 - 7850 h^2 + 86975 h \\
               & + 445850) \A^4 + (-h^4 + 10 h^3 + 790 h^2 - 8440 h \\
               & - 46680) \A^5 
                 \qquad\qquad\qquad\qquad\qquad\qquad\qquad~~~ 
                 \text{by \eqref{w}} 
  \end{split}
 \end{equation}
 We also have $\psi^5(c) = F c$ for $c \in \BW \big( \cF_5 \big)$, where $F$ is 
 the Frobenius automorphism.  As $\psi^5$ is a local ring homomorphism, 
 \eqref{psi5} then determines $\psi^5(x)$ for all 
 $x \in E^0 \cong \BW \big( \cF_5 \big) \lb h \rb$.  
\end{ex}

The previous example illustrates a general recipe for computing power operations 
on a Morava $E$-theory at height 2 and prime $p$, with a model based on the 
moduli problem $\big(\CP_N,[\G_0(p)]\big)$.  Crucial in this computation is an 
explicit expression for 
\begin{equation}
 \label{generalw}
 w(\A) = \A^{p + 1} + w_p \A^p + \cdots + w_1 \A + w_0 \in E^0[\A] 
\end{equation}
The coefficients have the following properties.  
\begin{itemize}
 \item For $i \neq 1$, each $w_i$ is divisible by $p$.  

 \item The coefficient $w_1$ is a lift of the Hasse invariant at $p$, so we can 
 always choose $h = -w_1$ and have $w(\A) \equiv \A (\A^p - h) \md p$.  

 \item We have $w_0 = \widetilde{\A} \cdot \A = \mu p$ for some unit 
 $\mu \in E^0$ that depends on the Frobenius endomorphism of the supersingular 
 elliptic curve (cf.~\cite[3.8]{mc1}).  
\end{itemize}
Computing $w(\A)$ appears to be essentially computing a ``modular equation'' 
(see, e.g., \cite[II.6]{Milne} and \cite{MO}, and compare the {\em canonical 
modular polynomials} in 
\href{http://magma.maths.usyd.edu.au/magma/handbook}{Magma}'s Modular Polynomial 
Databases).  

\begin{rmk}
 \label{rmk:invl}
 It does not work in general to compute $\psi^p(h)$ based solely on the 
 involution of $w(\A)$ and the relation 
 \begin{equation}
  \label{w0}
  \widetilde{\A} \cdot \A = w_0 
 \end{equation}
 The coefficients $w_i$, $2 \leq i \leq p$ may involve $h$ (in terms of 
 functions of $a$).  For example, working with $\CP_3$ instead of $\CP_4$ for 
 the prime 5, we have 
 \[
  w(\A) = \A^6 - 5 a \A^4 + 40 \A^3 - 5 a^2 \A^2 - h \A - 5 
 \]
 with $h = -a^4 + 19 a$.  In this case, in order to derive $\psi^p(h)$, it seems 
 unavoidable to compute an explicit equation for the quotient curve 
 $\CC_N^{(p)}$ (see, e.g., \cite[proof of Proposition 2.3]{p3}).  
\end{rmk}

\subsection{Hecke operators}
\label{subsec:ho}

We first describe the classical action of Hecke operators on modular forms in 
terms of isogenies between elliptic curves.  The $p$'th Hecke operator $T_p$ 
that acts on $\MF\big(\G_1(N)\big)$ can be built from universal isogenies as 
follows.  

\begin{cstr}[{cf.~\cite[(1.11.0.2)]{padicprop}}]
 Let the notation be as in Section \ref{subsec:po}, with the subscripts ``$N$'' 
 suppressed.  Given any $f \in \MF\big(\G_1(N)\big)$ of weight $k \geq 1$, 
 $T_p~f \in \MF\big(\G_1(N)\big)$ is a modular form such that 
 \begin{equation}
  \label{Tp}
  T_p~f(\CC_S, P_0, du) \ce \frac{1}{p} \sum_{i = 0}^p \K_i^k \cdot 
  f\!\left(\CC_{S^{(p)}}/\CG^{(p)}_i, \Psi^{(p)}_i(P_0), du\right) 
 \end{equation}
 where each $\CG^{(p)}_i$ denotes a degree-$p$ subgroup of $\CC$ over $S^{(p)}$, 
 and $\Psi^{(p)}_i$ is the quotient map with kernel $\CG^{(p)}_i$ as in 
 Construction \ref{cstr}.  
\end{cstr}

\begin{rmk}
 \label{rmk:normalizing}
 The terms $\K_i^k$ appear in the above formula so that $T_p$ is independent of 
 the choice of a basis for the cotangent space.  Explicitly, we have 
 \[
  \big( \Psi^{(p)}_i \big)^* du = \K_i \cdot du \qquad \ad \qquad 
  \big( \Psi^{(p)}_i \big)^* 
  \big( \widecheck{\Psi}^{\raisebox{-5pt}{$\scriptstyle(p)$}}_i \big)^* du = 
  p \cdot du 
 \]
 where each $\widecheck{\Psi}^{\raisebox{-5pt}{$\scriptstyle(p)$}}_i \co 
 \CC/\CG^{(p)}_i \to \CC$ is a dual isogeny, and 
 $\big( \widecheck{\Psi}^{\raisebox{-5pt}{$\scriptstyle(p)$}}_i \big)^* du$ is 
 the choice in \cite{padicprop} for a nonvanishing one-form on a quotient curve 
 (cf.~Remark \ref{rmk:K} and the discussion before (1.11.0.0) in 
 \cite[Section 1.11]{padicprop}).  Thus we rewrite \eqref{Tp} as 
 \begin{equation*}
  \begin{split}
     & ~ \frac{1}{p} \sum_{i = 0}^p \K_i^k \cdot 
       f\!\left(\CC_{S^{(p)}}/\CG^{(p)}_i, \Psi^{(p)}_i(P_0), du\right) \\
   = & ~ \frac{1}{p} \sum_{i = 0}^p \K_i^k \cdot 
       f\!\left(\CC_{S^{(p)}}/\CG^{(p)}_i, \Psi^{(p)}_i(P_0), (\K_i/p) 
       \big(\widecheck{\Psi}^{\raisebox{-4pt}{$\scriptstyle(p)$}}\big)^* 
       du\right) \\
   = & ~ \frac{1}{p} \sum_{i = 0}^p p^k \cdot 
       f\!\left(\CC_{S^{(p)}}/\CG^{(p)}_i, \Psi^{(p)}_i(P_0), 
       \big(\widecheck{\Psi}^{\raisebox{-4pt}{$\scriptstyle(p)$}}\big)^* 
       du\right) 
  \end{split}
 \end{equation*}
 and this agrees with \cite[(1.11.0.2)]{padicprop}.  
\end{rmk}

\begin{cstr}[{cf.~\cite[1.12]{log}}]
 There is a {\em topological} Hecke operator $t_p \co E^0 \to p^{-1} E^0$ 
 defined by 
 \begin{equation}
  \label{tp}
  t_p(x) \ce \frac{1}{p} \sum_{i = 0}^p \psi^p_i(x) 
 \end{equation}
 where $\psi^p_i$ denotes the power operation 
 $\psi^p \co E^0 \to E^0(B\Sigma_p) / I \cong E^0[\A] / \big( w(\A) \big)$ with 
 the parameter $\A$ replaced by $\A_i = u^{-p + 1} \K_i$ (cf.~\eqref{A}).  
\end{cstr}

Since the parameters $\A_i$ are the roots of $w(\A) \in E^0[\A]$, $t_p$ indeed 
lands in $p^{-1} E^0$.  

\begin{rmk}
 \label{rmk:tc}
 Given the ring homomorphism $\B$ in Proposition \ref{prop:mfe0}, the 
 topological Hecke operator above does not coincide with the classical one on 
 modular forms.  Comparing \eqref{tp} to \eqref{Tp}, note that there are no 
 terms $\A_i^k$ corresponding to $\K_i^k$.  This is related to the fact that 
 $\B$ is not injective: if we include $\A_i^k$ in the definition, for each 
 $x \in E^0$ we need to determine a unique value of its ``weight'' $k$ so that 
 $t_p$ is well-defined.  

 Note that the map $\B$ is not surjective either, so an element in $E^0$ may not 
 come from any modular form.  On the other hand, the total power operation 
 $\psi^p$ (and hence $t_p$) is defined on the entire $E^0$.  
\end{rmk}

\begin{ex}
 \label{ex:ho}
 Again, let $p = 5$ and $N = 4$.  Recall from Example \ref{ex:mfe0} that we have 
 \[
  \B(\D) = \d = h - 26 
 \]
 In view of \eqref{w}, we then compute from \eqref{psi5} that 
 \begin{equation*}
  \begin{split}
   t_5(\d) = & ~ \frac{1}{5} \sum_{i = 0}^5 \psi^5_i(\d) \\
           = & ~ \frac{1}{5} \sum_{i = 0}^5 \big(\psi^5_i(h) - 26\big) \\
           = & ~ \frac{1}{5} (h^5 - 10 h^4 - 1340 h^3 + 18440 h^2 + 267430 h 
               - 3178396) \\
           = & ~ \frac{1}{5} (h^4 + 16 h^3 - 924 h^2 - 5584 h + 122246) \cdot \d 
  \end{split}
 \end{equation*}
 Given that the modular form $\D$ is of weight 12, we define and compute 
 \begin{equation*}
  \begin{split}
   \widecheck{t}_5(\d) \ce & ~ \frac{1}{5} \sum_{i = 0}^5 \A_i^{12} \cdot 
                             \psi^5_i(\d) \qquad\qquad\qquad\qquad\qquad\qquad
                             \qquad\qquad\quad \\
                         = & ~ 4830 h - 125580 \\
                         = & ~ \T(5) \cdot \d 
  \end{split}
 \end{equation*}
 where $\T \co \BN \to \BZ$ is the Ramanujan tau-function, with $\T(5) = 4830$.  
 This recovers the action of $T_5$ on $\D$.  
\end{ex}

More generally, in the theory of automorphic forms on 
$\G_1(N) \subset \SL_n(\BZ)$, the above Hecke operator $T_p$ can be renamed as 
$T_{1,\,p}$.  It belongs to a family of operators $T_{i,\,p}$, $1 \leq i \leq n$ 
that generate the $p$-primary Hecke algebra (see, e.g., 
\cite[Theorems 3.20 and 3.35]{AF}).  

For the case $n = 2$, the other Hecke operator $T_{2,\,p}$ arises from the 
isogeny of multiplication by $p$, whose kernel is the degree-$p^2$ subgroup 
consisting of the $p$-torsion points.  Explicitly (again with the subscripts 
``$N$'' suppressed), given any $f \in \MF\big(\G_1(N)\big)$ of weight 
$k \geq 2$, there exists $\l \in S^\times$ of degree $-p^2 + 1$ such that 
\begin{equation}
 \label{T2p}
 \begin{split}
  T_{2,\,p}~f(\CC, P_0, du) \ce & ~ \frac{1}{p^2} \Big(\! (p \l)^k \!\cdot 
                                  f\!\left(\CC/\CC[p], \, [p] P_0, \, 
                                  du\right) \!\!\Big) \\
                              = & ~ p^{k - 2} \!\cdot f\!\left(\CC/\CC[p], \, 
                                  [p] P_0, \, \l^{-1} du\right) \\
                              = & ~ p^{k - 2} \!\cdot f(\CC, [p] P_0, du) 
                                  \qquad\qquad\qquad~~ \text{via 
                                  $\CC/\CC[p] \xrightarrow{\sim} \CC$} 
 \end{split}
\end{equation}
For example, when $N = 4$ with $p = 3$ or $p = 5$, we have 
$\lambda = B^{(1 - p^2)/2}$ (cf.~Example \ref{ex:4}).  Via the 
$S^{(p)}$-isomorphism 
\begin{equation}
 \label{iso}
 \CC/\CC[p] \cong \frac{\CC/\CG^{(p)}}{\CC[p]/\CG^{(p)}} \eqqcolon 
 \frac{\CC^{(p)}}{\widetilde{\CG}^{(p)}} 
\end{equation}
the quotient curve $\CC/\CC[p]$ can be identified with the target in the 
composite 
\[
 \CC \xrightarrow{\Psi^{(p)}} \CC^{(p)} \xrightarrow{\widetilde{\Psi}^{(p)}} 
 \CC^{(p)}/\widetilde{\CG}^{(p)} 
\]
of deformations of Frobenius isogenies.  

Correspondingly, given any $K(2)$-local commutative $E$-algebra $A$, there is 
a composite $\f$ of total power operations 
$\psi^p \circ \psi^p \co A^0 \to A^0$.  In view of \eqref{psi}, note that $\phi$ 
lands in $A^0$, because up to the isomorphism \eqref{iso} the composite 
$\widetilde{\Psi}^{(p)} \circ \Psi^{(p)}$ is an endomorphism on $\CC$ over $S$.  
In particular, on $E^0$ or on $E^0(B\Sigma_p) / I$, the operation $\phi$ is the 
identity map, which is a manifest of the Atkin-Lehner involution (cf.~Example 
\ref{ex:po}).  

Taking $A = E$, we define a {\em topological} Hecke operator 
$t_{2,\,p} \co E^0 \to p^{-1} E^0$ by 
\begin{equation}
 \label{t2p}
 t_{2,\,p}(x) \ce p^{-2} \phi(x) = p^{-2} x 
\end{equation}
(we will discuss the general case on $A^0$ in Section \ref{sec:individual}).  
Rewrite $t_{1,\,p} \ce t_p$.  The ring $\BZ[t_{1,\,p},t_{2,\,p}]$ then acts on 
$p^{-1} E^0$.  As we see in Remark \ref{rmk:tc}, along the map 
$\B \co \MF\big(\G_1(N)\big) \to E^0$, this action is not compatible with the 
action of the Hecke algebra $\BZ[T_{1,\,p},T_{2,\,p}]$ on modular forms.  
However, in certain instances, the two actions do interact well.  The next 
section will exploit this connection.

\section{Kernels of logarithmic cohomology operations at height 2}
\label{sec:kerlog}

In Section \ref{subsec:ho}, we give an explicitation for an algebra of 
topological Hecke operators acting on $E^0$, where $E$ is a Morava $E$-theory of 
height 2.  Based on this, we now use a number-theoretic calculation to 
understand a topological construction.  This is Rezk's construction of 
logarithmic cohomology operations via Bousfield-Kuhn functors \cite{log}.  We 
first recall a formula of Rezk for computing these operations, which leads to a 
connection with Hecke operators.

\subsection{Connecting to Hecke operators}

Given any $E_\infty$-ring spectrum $R$, let $\gl_1 R$ be the spectrum of units 
of $R$.  Rezk constructs a family of operations that naturally acts on $\gl_1 R$ 
\cite[Definition 3.6]{log}.  Specifically, given a positive integer $n$ and a 
prime $p$, let $L_{K(n)}$ denote localization with respect to the $n$'th Morava 
$K$-theory at $p$, and let $\Phi_n$ be the corresponding Bousfield-Kuhn functor 
from the category of based topological spaces to the category of spectra.  
Consider the composite 
\begin{equation}
 \label{log}
 \gl_1 R \to L_{K(n)} \gl_1 R \simeq \Phi_n \Omega^\infty \gl_1 R 
 \xrightarrow{\sim} \Phi_n \Omega^\infty R \simeq L_{K(n)} R 
\end{equation}
Note that $\Omega^\infty \gl_1 R$ and $\Omega^\infty R$ have weakly equivalent 
basepoint components, but the standard inclusion 
$\Omega^\infty \gl_1 R \hookrightarrow \Omega^\infty R$ is not 
basepoint-preserving.  The equivalence 
$\Phi_n \Omega^\infty \gl_1 R \to \Phi_n \Omega^\infty R$ thus involves a 
``basepoint shift'' (see \cite[3.4]{log}).  

Let $E$ be a Morava $E$-theory of height $n$ at the prime $p$.  With $R = E$, 
applying $\pi_0(-)$ to \eqref{log}, we then obtain the logarithmic operation 
\[
 \ell_{n,\,p} \co \big( E^0 \big)^{\!\times} \to E^0 
\]
which is a homomorphism from a multiplicative group to an additive group.  More 
generally, let $X$ be a space.  Taking $R$ to be the spectrum of functions from 
$\Sigma^\infty_+ X$ to $E$, we obtain the operation $\ell_{n,\,p}$ that acts on 
$E^0(X)^\times$, naturally both in $X$ and in $E$.  

The main theorem of \cite{log} is a formula for this operation 
\cite[Theorem 1.11]{log}.  In particular, for any 
$x \in \big( E^0 \big)^{\!\times}$, 
\begin{equation}
 \label{M}
 \ell_{n,\,p}(x) = \frac{1}{p} \log\big(1 + p M(x)\big) 
\end{equation}
where $M \co \big( E^0 \big)^{\!\times} \to E^0$ is a cohomology operation that 
can be expressed in terms of power operations $\psi_A$ associated to certain 
subgroups $A$ of $(\BQ_p/\BZ_p)^n$.  Explicitly, 
\[
 1 + p M(x) = \prod_{j=0}^n ~ \prod_{\stackrel{\scriptstyle A \subset 
 (\BQ_p/\BZ_p)^n [p]}{\#A \, = \, p^{\,j}}} 
 \psi_A(x)^{(-1)^{\,j} p^{(j-1)(j-2)/2}} 
\]

Now let the $E$-theory be of height $n = 2$.  In the presence of a model for 
$E$, the operations $\psi_A$ above coincide with the power operations in Section 
\ref{sec:ct}.  In particular, 
\begin{equation}
 \label{l2p}
 \ell_{2,\,p}(x) = \frac{1}{p} \log \left( x^p \cdot 
 \frac{1}{\psi^p_0(x) \cdots \psi^p_p(x)} \cdot \phi(x) \right) 
\end{equation}
Note that $\ell_{2,\,p}(x) = 0$ for any 
$x \in \BZ_p \cap \big( E^0 \big)^{\!\times}$.  

Suppose that $x = \B(\,f)$ for some modular form $f$ of weight $k$ (see 
Proposition \ref{prop:mfe0}).  We next write \eqref{l2p} in terms of Hecke 
operators.  We note the following.  
\begin{itemize}
 \item Recall from \eqref{generalw} that the parameters $\A_i$ in the operations 
 $\psi^p_i$ satisfy 
 \[
  \A_0 \cdots \A_p = (-1)^{p+1} w_0 = (-1)^{p+1} \mu p 
 \]
 for some $\mu \in \big( E^0 \big)^{\!\times}$.  In fact, by a theorem of Ando 
 \cite[Theorem 4]{Ando95}, there exists a unique coordinate on the formal group 
 of $E$ such that its corresponding parameters $\A_i$ satisfy 
 $\A_0 \cdots \A_p = p$.  This is the coordinate $u = x/y$ when $p = 2$ with the 
 $\CP_3$-model \cite[Section 3]{h2p2} and when $p = 5$ with the $\CP_4$-model 
 \eqref{w}.  Henceforth we will choose this particular coordinate for all primes 
 $p$ and denote it by $u$.  

 \item Recall that $\psi^p_i$, $\phi$ and $\B$ are ring homomorphisms.  
 Multiplications by the terms $\K_i^k$ in \eqref{Tp} and by $p^k$ in \eqref{T2p} 
 can be viewed as ring homomorphisms as well.  More precisely, we allow the 
 exponent $k$ to vary, according to the weight of the modular form that 
 immediately follows.  We will denote these ring homomorphisms by $\K_i^\bullet$ 
 and $p^\bullet$.  The topological analog of $\K_i^\bullet$ on homotopy groups 
 is $\A_i^\bullet$.  In particular, given as a formal power series, the usual 
 logarithm ``$\log$'' commutes with these continuous ring homomorphisms.  
\end{itemize}
In Section \ref{subsec:ho}, we give a comparison between classical and 
topological Hecke operators, particularly as illustrated in Example \ref{ex:ho}.  
In view of this comparison and the two observations above, we can now rewrite 
\eqref{l2p} as follows (cf.~\cite[1.12]{log}).  Given $x = \B(\,f)$, 
\begin{equation}
 \label{grdl2p}
 \begin{split}
  \ell_{2,\,p}(x) = & ~ \frac{1}{p} \log \frac{x^p \cdot p^\bullet 
                      \phi(x)}{\prod_{i=0}^p \A_i^\bullet \psi^p_i(x)} \\
                  = & \left( 1 - \frac{1}{p} \sum_{i = 0}^p \A_i^\bullet 
                      \psi^p_i + p \cdot \frac{1}{p^2} p^\bullet \phi \right) 
                      \log x \\
                  = & ~ \B \left( (1 - T_{1,\,p} + p \cdot T_{2,\,p}) 
                      \log f \right) 
 \end{split}
\end{equation}
Following Rezk \cite[1.12]{log}, we write 
\begin{equation}
 \label{FX}
 F_X \ce 1 - T_{1,\,p} \cdot X + p T_{2,\,p} \cdot X^2 \in 
 \BZ[T_{1,\,p},T_{2,\,p}][X] 
\end{equation}
In particular, with $X = 1$, \eqref{grdl2p} becomes 
\begin{equation}
 \label{F1}
 \ell_{2,\,p}(x) = \B \big( F_1 (\log f) \big) 
\end{equation}

\begin{rmk}
 \label{rmk:ratio}
 Let $x = \B(\,f) \in \big( E^0 \big)^{\!\times}$ for some unit $f$ in 
 $\MF\big(\G_1(N)\big)$ of weight $k$ (cf.~\cite{KubertLang}).  The formula 
 \eqref{l2p} expresses $p \ell_{2,\,p}(x)$ as the logarithm of a ratio: the 
 numerator, via $\B$, corresponds to a modular form of weight $p k + p^2 k$, and 
 the denominator corresponds to one having weight $(p + 1) p k = p k + p^2 k$.  
 Thus $p \ell_{2,\,p}(x)$ corresponds to a $p$-adic modular form of weight 0 in 
 view of \eqref{M} (cf.~\cite[Section 10.1]{padicinterp}).  We note the 
 similarity between the logarithm in \cite[10.2.7]{padicinterp} and the one 
 following Theorem 1.9 in \cite{log}.  
\end{rmk}

\begin{ex}
 \label{ex:log}
 We revisit the case $p = 5$, with the model for the $E$-theory given by the 
 moduli problem $\CP_4$.  Consider $\d = \B(\D) \in \big( E^0 \big)^{\!\times}$.  
 As in Example \ref{ex:ho}, we calculate from \eqref{psi5} and \eqref{w} that 
 \begin{equation*}
  \begin{split}
   \ell_{2,\,5}(\d) = & ~ \frac{1}{5} \log \left( \d^5 \cdot 
                        \frac{1}{\psi^p_0(\d) \cdots \psi^p_p(\d)} \cdot 
                        \phi(\d) \right) \\
                    = & ~ \frac{1}{5} \log \left( \d^5 \cdot \frac{1}{\d^6} 
                        \cdot \d \right) \\
                    = & ~ \frac{1}{5} \log 1 \\
                    = & ~ 0 
  \end{split}
 \end{equation*}
 We compare this to the prime-2 case with the $\CP_3$-model.  In 
 \cite[Proposition 3.2]{tmf3}, Mahowald and Rezk show that $\CP_3$ is 
 represented by 
 \[
  y^2 + A x y + B y = x^3 
 \]
 over $\BZ[1/3][A, B, \D^{-1}]$ with $|A| = 1$, $|B| = 3$, and 
 $\D = B^3 (A^3 - 27 B)$.  In \cite[2.8]{h2p2}, Rezk computes that 
 \[
  \ell_{2,\,2}\big(\B_{2,\,3}(\D)\big) = \frac{1}{2} \log(-1) = 0 
 \]
 We interpret the ``$\log$'' here as the 2-adic logarithm (see, e.g., 
 \cite[\S IV.1]{padic} and cf.~\cite[10.2.16]{padicinterp}).  Again, the modular 
 form $\D$ produces an element in the kernel of a logarithmic operation.  

 Moreover, since $\D = B^3 (A^3 - 27 B)$ in this case, we have 
 \[
  \B_{2,\,3}(\D) = (a - 3) (a^2 + 3 a + 9) 
 \]
 with $E^0 \cong \BW \big( \cF_2 \big) \lb a \rb$ (cf.~\cite[Section 4]{h2p2}).  
 Using Rezk's formula for the total power operation on $E^0$, we compute that 
 \[
  \ell_{2,\,2}(a - 3) = \frac{1}{2} \log(-1) = 0 
 \]
 \[
  \ell_{2,\,2}(a^2 + 3 a + 9) = \frac{1}{2} \log 1 = 0 
 \]
 Thus, as $\ell_{2,\,2}$ is a group homomorphism, the formal power series 
 \[
  \frac{1}{a - 3} = -\sum_{i = 0}^\infty 
                    \Big( \sum_{j = 0}^\infty (-2)^{\,j} \Big)^{\!i + 1} a^i ~ 
                    \in \BW \big( \cF_2 \big) \lb a \rb^\times 
 \]
 is also contained in the kernel of $\ell_{2,\,2}$.  

 The above turn out to be instances of a general vanishing result for the 
 logarithms that we discuss next.  
\end{ex}

\subsection{Detecting the kernels}
\label{subsec:thm}

In this section, via the formula \eqref{grdl2p} for a logarithmic operation in 
terms of Hecke operators, we detect a family of elements contained in the kernel 
of this operation.  It includes the elements computed in Example \ref{ex:log}.  

We first review some preliminaries about differential structures on rings of 
modular forms (see, e.g., \cite[\S 5]{1-2-3} and \cite[Section 2.3]{web}).  In 
connection to Morava $E$-theories, we have been considering the $p$-local 
behavior of integral modular forms of level $\G_1(N)$, with $p$ not dividing $N$ 
(which embed into the ring of $p$-adic modular forms).  Thus for our purpose 
these modular forms can equivalently be viewed as over $\BC$.  Henceforth we 
will freely interchange between these two perspectives.  

Recall that there is a differential operator $D$ acting on meromorphic elliptic 
functions over $\BC$.  Specifically, given any meromorphic modular form $f$, it 
has a $q$-expansion 
\[
 f(z) = \sum_{j > -\infty} a_j \, q^{\,j} 
\]
with $a_j \in \BC$, where $q = e^{2 \pi i z}$ as usual.  We then have 
\begin{equation}
 \label{ramanujan}
 D ~\! f \ce \frac{1}{2 \pi i} \frac{df}{dz} = q \frac{df}{dq} 
\end{equation}
If the $q$-expansion of $f$ has coefficients in $\BZ$, so does the $q$-series 
for $D ~\! f$.  

In general, the function $D ~\! f$ is no longer modular.  There is another 
derivation $\vartheta$ that preserves modularity.  If $f$ has weight $k$, its 
{\em Serre derivative} is defined by 
\begin{equation}
 \label{serre}
 \vartheta ~\! f \ce D ~\! f - \frac{k}{12} \CE_2 \cdot f 
\end{equation}
where $\CE_2(z) = 1 - 24 \sum_{j = 1}^\infty \si_1(\,j\,) \, q^{\,j}$ is the 
quasimodular Eisenstein series of weight 2.  This is a meromorphic modular form 
of weight $k + 2$ (see \cite[Proposition 2.11]{web}, and 
cf.~\cite[Th\'eor\`eme 5(a)]{fmpadiq} for the action of $D$ on $p$-adic modular 
forms).  

\begin{ex}
 \label{ex:D}
 Consider the modular discriminant $\D$, a cusp form of weight 12.  The product 
 expansion 
 \[
  \D = q \prod_{n=1}^\infty (1 - q^n)^{24} 
 \]
 implies that 
 \begin{equation}
  \label{logD}
  \begin{split}
   \log \D = & ~ \log q + 24 \sum_{n = 1}^\infty \log(1 - q^n) \\
           = & ~ \log q + 24 \sum_{n = 1}^\infty \sum_{k = 1}^\infty 
               (-1)^{k - 1} \frac{(-q^n)^k}{k} \\
           = & ~ \log q - 24 \sum_{m = 1}^\infty \si_{-1}(m) \, q^{m} 
  \end{split}
 \end{equation}
 Thus $D \log \D = \CE_2$, and hence $\vartheta \D = 0$.  
\end{ex}

\begin{thm}
 \label{thm:kerlog}
 Let $E$ be a Morava $E$-theory of height $2$ at the prime $p$, and let $N > 3$ 
 be any integer prime to $p$.  Let 
 \[
  \ell_{2,\,p} \co \big( E^0 \big)^{\!\times} \to E^0 
 \]
 be Rezk's logarithmic cohomology operation.  Let 
 \[
  \B \co \MF\big(\G_1(N)\big) \to E^0 
 \]
 be the ring homomorphism in Proposition \ref{prop:mfe0}.  Suppose that \!\, 
 $f \in \MF\big(\G_1(N)\big)^{\!\times}$ has trivial Nebentypus character.  If 
 the Serre derivative $\vartheta ~\! f = 0$, then $\B(\,f)$ is contained in the 
 kernel of $\ell_{2,\,p}$.  
\end{thm}

Our proof consists of two parts.  In the first part (Lemma \ref{lem:const} 
below), we show that $\ell_{2,\,p}\big(\B(\,f)\big)$ is constant, i.e., it is 
the image of a constant modular form under $\B$.  This is based on an interplay 
between the differential structures and the action of Hecke operators on modular 
forms.  Looking at $q$-expansions, we then show in the second part that this 
constant equals zero.  It boils down to an analysis of the behavior of Tate 
curves under isogenies.  

\begin{lem}
 \label{lem:const}
 Given any $f \in \MF\big(\G_1(N)\big)^{\!\times}$ such that 
 $\vartheta ~\! f = 0$, the elliptic function $F_1(\log f)$ is constant, where 
 $F_1 = 1 - T_{1,\,p} + p T_{2,\,p}$ is the operator defined in \eqref{FX}.  
\end{lem}

\begin{proof}
 Suppose that $f$ is of weight $k$.  By \eqref{serre}, since 
 $\vartheta ~\! f = 0$, we have 
 \[
  D \log f = \frac{D ~\! f}{f} = \frac{k \CE_2}{12} 
 \]
 Note that $\CE_2$ is a weight-2 eigenform for every Hecke operator 
 $T_m = T_{1,\,m}$, $m \geq 1$ with eigenvalue $\si_1(m)$.  

 By comparing the effects on $q$-expansions, we have 
 \begin{equation}
  \label{DT}
  D \circ T_{i,\,p} = \frac{1}{p^i} \cdot T_{i,\,p} \circ D 
 \end{equation}
 for $i = 1$ and $i = 2$.  We then compute that 
 \begin{equation*}
  \begin{split}
   D \big( F_1(\log f) \big) = & ~ D \big( (1 - T_{1,\,p} + p T_{2,\,p}) 
                                 \log f \big) \\
                             = & \left( 1 - \frac{1}{p} \cdot T_{1,\,p} + 
                                 \frac{1}{p^2} \cdot p T_{2,\,p} \right) 
                                 (D \log f) \\
                             = & \left( 1 - \frac{1}{p} \cdot T_{1,\,p} + 
                                 \frac{1}{p^2} \cdot p T_{2,\,p} \right) 
                                 \frac{k \CE_2}{12} \\
                             = & \left( 1 - \frac{1}{p} \cdot (1 + p) + 
                                 \frac{1}{p^2} \cdot p \Big( \frac{1}{p^2} \cdot 
                                 p^2 \Big) \!\! \right) \frac{k \CE_2}{12} \\
                             = & ~ 0 
  \end{split}
 \end{equation*}
 Thus as a function of the complex variable $z$, $F_1(\log f)$ is constant.  
\end{proof}

\begin{proof}[Proof of Theorem \ref{thm:kerlog}]
 Let $x \ce \B(\,f)$.  Since $\B$ is a ring homomorphism, 
 $x \in \big( E^0 \big)^{\!\times}$.  Recall from \eqref{grdl2p} that 
 \begin{equation}
  \label{grdl2pbis}
  \ell_{2,\,p}(x) = \frac{1}{p} \log 
  \frac{x^p \cdot p^\bullet \phi(x)}{\prod_{i=0}^p \A_i^\bullet \psi^p_i(x)} = 
  \B \big( F_1(\log f) \big) 
 \end{equation}
 By Lemma \ref{lem:const}, the above equals an element in 
 $\BW \big( \cF_p \big)^{\!\times}$, which we denote by $c_f$.  

 Suppose that $f$ has weight $k$.  Via the correspondence between power 
 operations and deformations of Frobenius in \cite[Theorem B]{cong}, the ratio 
 of values of power operations on $x$ in \eqref{grdl2pbis} equals a ratio of 
 values of $f$ on the corresponding universal elliptic curves.  Explicitly, with 
 notation as in \eqref{Tp} and \eqref{T2p}, we have 
 \begin{equation}
  \label{ratio}
  \frac{x^p \cdot p^\bullet \phi(x)}{\prod_{i=0}^p \A_i^\bullet \psi^p_i(x)} = 
  \frac{f(\CC_S, P_0, du)^p \cdot p^k f(\CC_S, [p] P_0, du)}{\prod_{i=0}^p 
  \K_i^k ~ f\!\big(\CC_{S^{(p)}}/\CG^{(p)}_i, \Psi^{(p)}_i(P_0), du\big)} 
 \end{equation}
 To determine $c_f$, we need only inspect the constant term in a $q$-expansion 
 for the right-hand side.  

 Over a punctured formal neighborhood of each cusp, the universal curve $\CC$ is 
 isomorphic to the Tate curve $\Tate(q^N)$ with the level-$\G_1(N)$ structure 
 corresponding to that cusp.  The universal degree-$p$ isogeny on $\Tate(q^N)$ 
 is defined over the ring $\BZ [1/pN, \zeta_p] \lp q^{1/p} \rp$, where $\zeta_p$ 
 is a primitive $p$'th root of unity (see 
 \cite[Sections 1.2, 1.4, and 1.11]{padicprop}).  In particular, the $(p + 1)$ 
 subgroups of order $p$ are 
 \[
  \qquad~~ \text{$G^{(p)}_0$,~~~generated by $\zeta_p$,} \qquad \ad \qquad 
  \text{$G^{(p)}_i$, $1 \leq i \leq p$,~~~generated by $(\zeta_p^i q^{1/p})^N$} 
 \]
 Let $\sum_{j = m}^\infty a_j \, q^{\,j}$ be a $q$-expansion of $f$.  We now 
 compare as follows the lowest powers of $q$ at the denominator and the 
 numerator of \eqref{ratio}.  

 By \cite[(1.11.0.3) and (1.11.0.4)]{padicprop} (cf.~Section 
 \ref{subsec:correction}), including a normalizing factor $p^k$ in place of 
 $\K_i^k$ (see Remark \ref{rmk:normalizing}), we have at the denominator a 
 leading term as the product of 
 \[
  p^k a_m (q^p)^m \qquad \ad \qquad p^k \left( p^{-k} a_m \big( \zeta_p^i 
  q^{1/p} \big)^m \right) 
 \]
 where $i$ runs from 1 to $p$.  Note that since the Nebentypus character of $f$ 
 is trivial, the coefficient $a_m$ is independent of where the level-$\G_1(N)$ 
 structure goes under each degree-$p$ isogeny.  

 At the numerator, for the first factor, we have a leading term $(a_m q^m)^p$.  
 For the second factor, we have a leading term $p^k a_m q^m$, again under the 
 assumption that $f$ has trivial Nebentypus character.  

 Combining these, we see that the ratio has a leading constant term 
 \[
  \frac{(a_m q^m)^p \cdot p^k a_m q^m}{p^k a_m (q^p)^m \cdot \prod_{i=1}^p p^k 
  \Big( p^{-k} a_m \big( \zeta_p^i q^{1/p} \big)^m \Big)} = 
  \zeta_p^{-m (1 + p) p / 2} = \left\{\!\!
  \begin{array}{cl}
    (-1)^m & \quad \text{if $p = 2$} \\
    1 & \quad \text{if $p$ is odd} 
  \end{array}
  \right.
 \]
 Thus applying the $p$-adic logarithm (cf.~Example \ref{ex:log}) we get 
 $c_f = 0$.  
\end{proof}

\begin{rmk}
 Let $f$ be any nonzero meromorphic modular form on $\SL_2(\BZ)$.  Bruinier, 
 Kohnen, and Ono give an explicit formula for $\vartheta ~\! f$ as a multiple of 
 $-f$ by a certain function $f_\Theta$ \cite[Theorem 1]{BKO}.  The function 
 $f_\Theta$ encodes a sequence of modular functions $j_m(z)$, $m \geq 0$ defined 
 by applying Hecke operators to the usual $j$-invariant, whose values are of 
 arithmetic and combinatorial significance.  

 In particular, the formula of Bruinier, Kohnen, and Ono immediately shows that 
 a nonzero meromorphic modular form $f$ has vanishing Serre derivative precisely 
 when its zeros and poles are located only at the cusp 
 (cf.~\cite[Proposition 6]{DumasRoyer}).  The functions in Example \ref{ex:log}, 
 and in fact any unit in $\MF\big(\G_1(N)\big)$, all have the latter property, 
 though they are associated to $\G_1(N)$ instead of $\SL_2(\BZ)$.  

 Their theorem has been generalized by Ahlgren to $\G_0(p)$ with 
 $p \in \{2,3,5,7,13\}$ \cite[Theorem 2]{Ahlgren} and further by Choi to 
 $\G_0(n)$ for any $n$ that is square-free \cite[Theorem 3.4]{Choi}.  In view of 
 the assumption on Nebentypus character in Theorem \ref{thm:kerlog}, we note 
 that modular forms of level $\G_1(N)$ with trivial Nebentypus character are 
 precisely those of level $\G_0(N)$.  
\end{rmk}

\section{Extending the action of Hecke operators onto logarithmic $q$-series}
\label{sec:logq}

The purpose of this section is to give an account of elliptic functions of the 
form $\log f$, where $f$ is a meromorphic modular form.  As we have seen in 
Section \ref{sec:kerlog}, such functions arise in computations for logarithmic 
cohomology operations on Morava $E$-theories at height 2, e.g., in 
\eqref{grdl2p}.  

\begin{ex}
 \label{ex:logD}
 Consider the function $\log \D$ in Example \ref{ex:D}.  In the final form of 
 \eqref{logD}, the second summand is a convergent $q$-series.  We may call it an 
 {\em Eisenstein series of weight $0$}, by analogy to $q$-expansions for the 
 usual Eisenstein series of higher weight (also cf.~the real analytic Eisenstein 
 series of weight 0 discussed in \cite[Sections 3.3 and 4.1]{Funke}).  

 The first summand, $\log q$, never shows up in the $q$-expansion of a 
 meromorphic modular form.  It is indeed this term that we will address, given 
 the importance of $\D$ in the context of logarithmic operations (see Theorem 
 \ref{thm:kerlog}).  Specifically, with motivations from homotopy theory, we aim 
 to extend the classical action of Hecke operators on modular forms to 
 incorporate series such as \eqref{logD}.  We will then apply the extended 
 action in Example \ref{ex:pf}.  
\end{ex}

\begin{rmk}
 \label{rmk:eichler}
 Given a cusp form $f = \sum_{n = 1}^\infty a_n q^n$ of weight $k$, its Eichler 
 integral $\widetilde{f\,} = \sum_{n = 1}^\infty n^{-k + 1} a_n q^n$ is a mock 
 modular form of weight $2 - k$ (see the end of \cite[\S 6]{mock}).  Recall from 
 Example \ref{ex:D} that $D \log \D = \CE_2$.  Since 
 $D^{k - 1} \, \widetilde{f\,} = f$, we may then view $\log \D$ as a generalized 
 Eichler integral, ``generalized'' in that $\CE_2$ is not a cusp form.  We may 
 even approach proving Lemma \ref{lem:const} from this viewpoint.  
\end{rmk}

In \cite{KnoppMason}, given a representation 
$\rho \co \SL_2(\BZ) \to \GL_n(\BC)$, Knopp and Mason consider $n$-dimensional 
vector-valued modular forms associated to $\rho$.  In particular, they show in 
\cite[Theorem 2.2]{KnoppMason} that the components of certain vector-valued 
modular forms are functions 
\begin{equation}
 \label{logq}
 f(z) = \sum_{j = 0}^t (\log q)^{\,j} ~\! h_j(z) 
\end{equation}
where $t \geq 0$ is an integer, and each $h_j$ is a convergent $q$-series with 
at worst real exponents 
(cf.~\cite[(7), (13), and Sections 3.2--3.3]{KnoppMason}).  They remark that 
$q$-expansions of this form occur in logarithmic conformal field theory (e.g., 
cf.~\cite[(5.3.9)]{Zhu} and \cite[(6.12)]{DongLiMason}).  The one in 
\eqref{logD} gives another example.  Following Knopp and Mason, we call the 
series in \eqref{logq} a {\em logarithmic} $q$-series.  

\begin{prop}
 \label{prop:logq}
 The action of the Hecke operator $T_p$ on modular forms of level $\G_0(N)$ 
 extends so that 
 \[
  T_p (\log q) = \big( p^{-1} + p^{-2} \big) \log q 
 \]
\end{prop}

\begin{proof}
 Consider modular forms in $\MF\big(\G_1(N)\big)$ with trivial Nebentypus 
 character, i.e., those on $\G_0(N)$.  We follow the modular description in 
 \cite[Section 1.11]{padicprop} for Hecke operators in the presence of the Tate 
 curve $\Tate(q^N)$ over $\BZ [1/pN, \zeta_p] \lp q^{1/p} \rp$, where $\zeta_p$ 
 is a primitive $p$'th root of unity.  

 Write $\CF \ce \log q$.  Let $\o_\can$ be the canonical differential on 
 $\Tate(q^N)$.  Let $\CM_N \ce \Proj S_N$ be the scheme over $\BZ[1/N]$ 
 representing the moduli problem $\CP_N$ (see Examples \ref{ex:4}, \ref{ex:5}, 
 and \ref{ex:mfe0}).  Denote by $\ou \ce {\rm pr}_* \Omega^1_{\CC_N/\CM_N}$ the 
 pushforward along the structure morphism ${\rm pr} \co \CC_N \to \CM_N$ of the 
 relative cotangent sheaf $\Omega^1_{\CC_N/\CM_N}$.  

 By \cite[Theorem 10.13.11]{KM}, the isomorphism 
 \[
  \ou^{~\! \otimes 2} \xrightarrow{\sim} \Omega^1_{\CM_N \big/ \BZ[1/N]} 
 \]
 on $\CM_N$ extends to an isomorphism 
 \[
  \ou^{~\! \otimes 2} \xrightarrow{\sim} 
  \Omega^1_{\CMB_N \big/ \BZ[1/N]} (\text{log cusps}) 
 \]
 on the compactification $\CMB_N$, where the target is the invertible sheaf of 
 one-forms with at worst simple poles along the cusps.  In particular, over the 
 cusps, $\o_\can^2$ corresponds to $N \cdot d\CF$ under this isomorphism 
 (cf.~\cite[Section 1.5]{padicprop}).  Since $\CF = \log q = 2 \pi i z$ is 
 linear in $z$, we have 
 \begin{equation}
  \label{weight}
  \CF \big( \Tate(q^N), ~ P_0, ~ p \cdot \o_\can \big) = 
  p^2 \cdot \CF \big( \Tate(q^N), ~ P_0, ~ \o_\can \big) 
 \end{equation}
 By \cite[(1.11.0.3) and (1.11.0.4)]{padicprop} (cf.~Section 
 \ref{subsec:correction}), given the assumption of trivial Nebentypus character, 
 we then calculate that 
 \begin{equation}
  \label{Tplogq}
  \begin{split}
   T_p (\log q) = & ~ \frac{1}{p}\cdot p^k \left( \log(q^p) + \sum_{i = 1}^p 
                    p^{-k} \log(\zeta_p^i q^{1/p}) \right) \\
                = & ~ p^{k - 1} \left( p \log q + p^{-k} \sum_{i = 1}^p 
                    (\log \zeta_p^i + p^{-1} \log q) \right) \\
                = & ~ p^{k - 1} \left( p \log q + p^{-k} \sum_{i = 1}^p 
                    p^{-1} \log q \right) \\
                = & ~ p^{k - 1} \big( p + p^{-k} \big) \log q 
  \end{split}
 \end{equation}
 where we interpret ``$\log$'' as the $p$-adic logarithm so that 
 $\log \zeta_p^i = 0$ (cf.~Example \ref{ex:log}).  Setting $k = -2$ in view of 
 \eqref{weight}, we obtain the stated identity.  
\end{proof}

\begin{rmk}
 \label{rmk:nebentypus}
 More generally, if we view $\log q$ as generalizing modular forms in 
 $\MF\big(\G_1(N)\big)$ with Nebentypus character $\chi$, we compute as in 
 \eqref{Tplogq} and obtain 
 \[
  T_p (\log q) = \big( p^{-1} + \chi(p) \, p^{-2} \big) \log q 
 \]
 For the rest of this section, we focus on the case when $\chi$ is trivial 
 (cf.~Theorem \ref{thm:kerlog}).  
\end{rmk}

Let $K$ be a number field.  In \cite{BruinierOno}, Bruinier and Ono study 
meromorphic modular forms $g$ on $\SL_2(\BZ)$ with $q$-expansion 
\begin{equation}
 \label{product}
 g(z) = q^m \left( 1 + \sum_{n = 1}^\infty a_n ~\! q^n \right) 
\end{equation}
where $m \in \BZ$ and $a_n \in \CO_K$ (the ring of integers in $K$).  They show 
that if $g$ satisfies a certain condition with respect to a prime $p$, its 
logarithmic derivative $D \log(g)$ is a $p$-adic modular form of weight 2 
\cite[Theorem 1]{BruinierOno} (the differential operator $D$ is defined in 
\eqref{ramanujan}).  

This theorem has been generalized to meromorphic modular forms on $\G_0(p)$ for 
$p \geq 5$ \cite[Theorem 4]{Getz}.  Examples include $\CE_{p - 1}$ at each 
$p \geq 5$ and, for all $p$, meromorphic modular forms whose zeros and poles are 
located only at the cusps (cf.~\cite[Definition 3.1]{BruinierOno}).  In 
particular, when $g = \D$, we have $D \log \D = \CE_2$ (cf.~the discussion 
before \cite[Th\'eor\`eme 5]{fmpadiq}).  

Given any $g$ as in \eqref{product}, note that $\log(g)$ is a logarithmic 
$q$-series.  The following definition is based on Proposition \ref{prop:logq} 
(specifically \eqref{weight} and \eqref{Tplogq}) and the theorem of Bruinier and 
Ono above.  

\begin{defn}
 \label{def:logq}
 \mbox{}
 \begin{enumerate}[(i)]
  \item Given any integer $j \geq 0$, define the {\em weight} of 
  $(\log q)^{\,j}$ to be $-2\,j$.  

  \item \label{ii} Let $g$ be a meromorphic modular form such that $D \log(g)$ 
  is a $p$-adic modular form of weight 2 for all $p$.  Define the {\em weight} 
  of $\log(g)$ to be 0.  

  \item \label{iii} Given any prime $p$ and any logarithmic $q$-series 
  \[
   f = \sum_{j = 0}^t (\log q)^{\,j} ~\! h_j 
  \]
  of weight $k$, define $T_p~f$ as follows.  For each $j$, if 
  \[
   h_j = \sum_{m > -\infty} a_m ~\! q^m 
  \]
  (the index $m$ and the coefficients $a_m$ depend on $j$), define 
  \[
   T_p \! \left( (\log q)^{\,j} ~\! h_j \right) \ce (\log q)^{\,j} 
   \sum_{m > -\infty} b_m ~\! q^m 
  \]
  where the coefficients 
  \[
   b_m = p^{\,j + k - 1} a_{m/p} + p^{-j} a_{p m} 
  \]
  with the convention that $a_{m/p} = 0$ unless $p|m$.  We then define 
  \[
   T_p~f \ce \sum_{j = 0}^t T_p \! \left( (\log q)^{\,j} ~\! h_j \right) 
  \]
 \end{enumerate}
\end{defn}

\begin{rmk}
 \mbox{}
 \begin{enumerate}[(i)]
  \item The definitions of weight above are compatible with the action of the 
  differential operator $D$.  Specifically, we have 
  \[
   D \log q = q \frac{d}{dq} (\log q) = 1 
  \]
  Thus applying $D$ to $\log q$ increases the weight by 2, which agrees with 
  \cite[Th\'eor\`eme 5\,(a)]{fmpadiq}.  More generally, for $j \geq 0$, since 
  $D\,(\log q)^{\,j} = j\,(\log q)^{\,j - 1}$, the same compatibility holds.  

  \item The definition for $T_p \! \left( (\log q)^{\,j} ~\! h_j \right)$ 
  extends \cite[Formula 1.11.1]{padicprop} by a computation analogous to 
  \eqref{Tplogq} under the assumption of trivial Nebentypus character (see 
  Remark \ref{rmk:nebentypus}).  Moreover, the following identities for 
  operators acting on modular forms extend to the series 
  $(\log q)^{\,j} ~\! h_j$: 
  \begin{equation*}
   \begin{split}
         D \circ T_p = & ~ \frac{1}{p} \cdot T_p \circ D \qquad 
                       \text{cf.~\eqref{DT}} \\
    T_\ell \circ T_p = & ~ T_p \circ T_\ell \qquad\quad~ \text{for primes $\ell$ 
                       and $p$} 
   \end{split}
  \end{equation*}
  (assuming that Hecke operators preserve weight).  We can also define the Hecke 
  operators $T_m$ acting on $(\log q)^{\,j} ~\! h_j$ for any positive integer 
  $m$ as in the remark after \cite[Th\'eor\`eme 4]{fmpadiq}.  
 \end{enumerate}
\end{rmk}

\begin{ex}
 \label{ex:pf}
 We now return to Example \ref{ex:logD}.  By Definition 
 \ref{def:logq}\,\eqref{ii}, $\log \D$ is a logarithmic $q$-series of weight 0.  
 In view of \eqref{logD}, we then compute by Definition 
 \ref{def:logq}\,\eqref{iii} that 
 \[
  T_p (\log \D) = \si_{-1}(p) \log \D 
 \]
 Thus in \eqref{F1} we have 
 \[
  F_1 (\log \D) = \left( 1 - \si_{-1}(p) + p^{-1} \right) \log \D = 0 
 \]
 This gives a second proof of Theorem \ref{thm:kerlog} in the case $f = \D$.  
\end{ex}

\section{Topological Hecke operators in terms of individual power operations}
\label{sec:individual}

Constructed from total power operations, the topological Hecke operators in 
\eqref{tp} and \eqref{t2p} can be defined more generally on $E^0 X$ for any 
space $X$.  In this section, we examine their role in the algebra of additive 
$E$-cohomology operations, by expressing them in terms of generators of this 
algebra.  

Let $E$ be a Morava $E$-theory of height $n$ at the prime $p$.  There is an 
algebraic theory---in the sense of Lawvere \cite{Lawvere}---constructed from the 
extended power functors on the category of $E$-modules 
(cf.~\cite[Section 9]{lpo} and \cite[Section 4]{cong}).  It describes all 
homotopy operations on $K(n)$-local commutative $E$-algebras.  

Rezk shows that under a certain congruence condition, a model for this algebraic 
theory amounts to a $p$-torsion-free graded commutative algebra over an 
associative ring $\G$ \cite[Theorem A]{cong}.  Following Rezk, we call $\G$ the 
{\em \DL algebra} of the $E$-theory.  It is constructed in \cite[6.2]{cong} as a 
direct sum of the $E^0$-linear duals of the rings $E^0(B\Sigma_{p^k}) / I$, for 
all $k \geq 0$ (cf.~\eqref{psi}).  

\begin{ex}
 \label{ex:gamma}
 Using the formulas for $\widetilde{h}$ and $\widetilde{\A}$ in \eqref{psi5}, we 
 compute as in \cite[Proposition 3.6]{p3} and obtain a presentation for the \DL 
 algebra $\G$ at height 2 and prime 5.  Specifically, $\G$ is the following 
 graded twisted bialgebra over $E_0 \cong \BW \big( \cF_5 \big) \lb h \rb$ with 
 generators $Q_i$, $0 \leq i \leq 5$.  

 \begin{itemize}
  \item The ``twists'' are given in terms of {\em commutation relations}: 
  \begin{equation*}
   \begin{split}
    Q_i c = & ~ (F c) Q_i ~~~ \text{for $c \in \BW \big( \cF_5 \big)$ and all 
              $i$, with $F$ the Frobenius automorphism} ~~ \\
    Q_0 h = & ~ (h^5 - 10 h^4 - 1065 h^3 + 12690 h^2 + 168930 h - 1462250) Q_0 + (5 h^4 \\
            & - 50 h^3 - 3950 h^2 + 42200 h + 233400) Q_1 + (25 h^3 - 250 h^2 - 12875 h \\
            & + 104750) Q_2 + (125 h^2 - 1250 h - 30000) Q_3 + (625 h - 6250) Q_4 \\
            & + 3120 Q_5 \\
    Q_1 h = & ~ (-55 h^4 + 850 h^3 + 39575 h^2 - 608700 h - 1113524) Q_0 + (-275 h^3 \\
            & + 4250 h^2 + 122250 h - 1462250) Q_1 + (-1375 h^2 + 21250 h \\
            & + 233400) Q_2 + (-6875 h + 104750) Q_3 - 30000 Q_4 + (h - 6250) Q_5 \\
    Q_2 h = & ~ (60 h^4 - 775 h^3 - 45400 h^2 + 593900 h + 2008800) Q_0 + (300 h^3 \\
            & - 3875 h^2 - 144500 h + 1453876) Q_1 + (1500 h^2 - 19375 h \\
            & - 310000) Q_2 + (7500 h - 96600) Q_3 + 36000 Q_4 + 4320 Q_5 
   \end{split}
  \end{equation*}
  \begin{equation*}
   \begin{split}
    Q_3 h = & ~ (-35 h^4 + 400 h^3 + 27125 h^2 - 320900 h - 1418300) Q_0 + (-175 h^3 \\
            & + 2000 h^2 + 87500 h - 792000) Q_1 + (-875 h^2 + 10000 h + 196876) Q_2 \\
            & + (-4375 h + 50000) Q_3 - 21600 Q_4 - 1440 Q_5 \\
    Q_4 h = & ~ (10 h^4 - 105 h^3 - 7850 h^2 + 86975 h + 445850) Q_0 + (50 h^3 - 525 h^2 \\
            & - 25500 h + 215500) Q_1 + (250 h^2 - 2625 h - 58750) Q_2 + (1250 h \\
            & - 13124) Q_3 + 6250 Q_4 + 240 Q_5 \\
    Q_5 h = & ~ (-h^4 + 10 h^3 + 790 h^2 - 8440 h - 46680) Q_0 + (-5 h^3 + 50 h^2 + 2575 h \\
            & - 20950) Q_1 + (-25 h^2 + 250 h + 6000) Q_2 + (-125 h + 1250) Q_3 \\
            & - 624 Q_4 + 10 Q_5 
   \end{split}
  \end{equation*}

  \item The product structure is subject to {\em Adem relations}: 
  \begin{equation*}
   \begin{split}
    Q_1 Q_0 = & ~ 55 Q_0 Q_1 + (55 h - 300) Q_0 Q_2 + (55 h^2 - 300 h - 14250) Q_0 Q_3 + (55 h^3 \\
              & - 300 h^2 - 29375 h + 163750) Q_0 Q_4 + (55 h^4 - 300 h^3 - 44500 h^2 \\
              & + 328750 h + 3228750) Q_0 Q_5 + 275 Q_1 Q_2 + (275 h - 1500) Q_1 Q_3 \\
              & + (275 h^2 - 1500 h - 71250) Q_1 Q_4 + (275 h^3 - 1500 h^2 - 146875 h \\
              & + 818750) Q_1 Q_5 - 5 Q_2 Q_1 + 1375 Q_2 Q_3 + (1375 h - 7500) Q_2 Q_4 \\
              & + (1375 h^2 - 7500 h - 356250) Q_2 Q_5 - 25 Q_3 Q_2 + 6875 Q_3 Q_4 \\
              & + (6875 h - 37500) Q_3 Q_5 - 125 Q_4 Q_3 + 34375 Q_4 Q_5 - 625 Q_5 Q_4 \\
    Q_2 Q_0 = & -60 Q_0 Q_1 + (-60 h + 175) Q_0 Q_2 + (-60 h^2 + 175 h + 16250) Q_0 Q_3 \\
              & + (-60 h^3 + 175 h^2 + 32750 h - 138000) Q_0 Q_4 + (-60 h^4 + 175 h^3 \\
              & + 49250 h^2 - 276125 h - 3943750) Q_0 Q_5 - 300 Q_1 Q_2 + (-300 h \\
              & + 875) Q_1 Q_3 + (-300 h^2 + 875 h + 81250) Q_1 Q_4 + (-300 h^3 + 875 h^2 \\
              & + 163750 h - 690000) Q_1 Q_5 - 1500 Q_2 Q_3 + (-1500 h + 4375) Q_2 Q_4 \\
              & + (-1500 h^2 + 4375 h + 406250) Q_2 Q_5 - 5 Q_3 Q_1 - 7500 Q_3 Q_4 \\
              & + (-7500 h + 21875) Q_3 Q_5 - 25 Q_4 Q_2 - 37500 Q_4 Q_5 - 125 Q_5 Q_3 \\
    Q_3 Q_0 = & ~ 35 Q_0 Q_1 + (35 h - 50) Q_0 Q_2 + (35 h^2 - 50 h - 9600) Q_0 Q_3 + (35 h^3 \\
              & - 50 h^2 - 19225 h + 66250) Q_0 Q_4 + (35 h^4 - 50 h^3 - 28850 h^2 \\
              & + 132500 h + 2411875) Q_0 Q_5 + 175 Q_1 Q_2 + (175 h - 250) Q_1 Q_3 \\
              & + (175 h^2 - 250 h - 48000) Q_1 Q_4 + (175 h^3 - 250 h^2 - 96125 h \\
              & + 331250) Q_1 Q_5 + 875 Q_2 Q_3 + (875 h - 1250) Q_2 Q_4 + (875 h^2 \\
              & - 1250 h - 240000) Q_2 Q_5 + 4375 Q_3 Q_4 + (4375 h - 6250) Q_3 Q_5 \\
              & - 5 Q_4 Q_1 + 21875 Q_4 Q_5 - 25 Q_5 Q_2 \\
    Q_4 Q_0 = & -10 Q_0 Q_1 + (-10 h + 5) Q_0 Q_2 + (-10 h^2 + 5 h + 2750) Q_0 Q_3 
   \end{split}
  \end{equation*}
  \begin{equation*}
   \begin{split}
              & + (-10 h^3 + 5 h^2 + 5500 h - 16375) Q_0 Q_4 + (-10 h^4 + 5 h^3 + 8250 h^2 \\
              & - 32750 h - 705000) Q_0 Q_5 - 50 Q_1 Q_2 + (-50 h + 25) Q_1 Q_3 + (-50 h^2 \\
              & + 25 h + 13750) Q_1 Q_4 + (-50 h^3 + 25 h^2 + 27500 h - 81875) Q_1 Q_5 \\
              & - 250 Q_2 Q_3 + (-250 h + 125) Q_2 Q_4 + (-250 h^2 + 125 h \\
              & + 68750) Q_2 Q_5 - 1250 Q_3 Q_4 + (-1250 h + 625) Q_3 Q_5 - 6250 Q_4 Q_5 \\
              & - 5 Q_5 Q_1 \\
    Q_5 Q_0 = & ~ Q_0 Q_1 + h Q_0 Q_2 + (h^2 - 275) Q_0 Q_3 + (h^3 - 550 h + 1500) Q_0 Q_4 + (h^4 \\
              & - 825 h^2 + 3000 h + 71250) Q_0 Q_5 + 5 Q_1 Q_2 + 5 h Q_1 Q_3 + (5 h^2 \\
              & - 1375) Q_1 Q_4 + (5 h^3 - 2750 h + 7500) Q_1 Q_5 + 25 Q_2 Q_3 + 25 h Q_2 Q_4\\
              & + (25 h^2 - 6875) Q_2 Q_5 + 125 Q_3 Q_4 + 125 h Q_3 Q_5 + 625 Q_4 Q_5 
   \end{split}
  \end{equation*}

  \item The coproduct structure is given by {\em Cartan formulas}: 
  \begin{equation*}
   \begin{split}
    Q_0(x y) = & ~ Q_0(x) Q_0(y) - 5 \big( Q_1(x) Q_5(y) + Q_2(x) Q_4(y) + Q_3(x) Q_3(y) \\
               & + Q_4(x) Q_2(y) + Q_5(x) Q_1(y) \big) - 50 \big( Q_2(x) Q_5(y) + Q_3(x) Q_4(y) \\
               & + Q_4(x) Q_3(y) + Q_5(x) Q_2(y) \big) - 325 \big( Q_3(x) Q_5(y) + Q_4(x) Q_4(y) \\
               & + Q_5(x) Q_3(y) \big) - 1800 \big( Q_4(x) Q_5(y) + Q_5(x) Q_4(y) \big) \\
               & - 9350 Q_5(x) Q_5(y) \\
    Q_1(x y) = & ~ \big( Q_0(x) Q_1(y) + Q_1(x) Q_0(y) \big) + h \big( Q_1(x) Q_5(y) + Q_2(x) Q_4(y) \\
               & + Q_3(x) Q_3(y) + Q_4(x) Q_2(y) + Q_5(x) Q_1(y) \big) + (10 h - 5) \big( Q_2(x) Q_5(y) \\
               & + Q_3(x) Q_4(y) + Q_4(x) Q_3(y) + Q_5(x) Q_2(y) \big) + (65 h \\
               & - 50) \big( Q_3(x) Q_5(y) + Q_4(x) Q_4(y) + Q_5(x) Q_3(y) \big) + (360 h \\
               & - 325) \big( Q_4(x) Q_5(y) + Q_5(x) Q_4(y) \big) + (1870 h - 1800) Q_5(x) Q_5(y) \\
    Q_2(x y) = & ~ \big( Q_0(x) Q_2(y) + Q_1(x) Q_1(y) + Q_2(x) Q_0(y) \big) - 55 \big( Q_1(x) Q_5(y) \\
               & + Q_2(x) Q_4(y) + Q_3(x) Q_3(y) + Q_4(x) Q_2(y) + Q_5(x) Q_1(y) \big) + (h \\
               & - 550) \big( Q_2(x) Q_5(y) + Q_3(x) Q_4(y) + Q_4(x) Q_3(y) + Q_5(x) Q_2(y) \big) \\
               & + (10 h - 3580) \big( Q_3(x) Q_5(y) + Q_4(x) Q_4(y) + Q_5(x) Q_3(y) \big) + (65 h \\
               & - 19850) \big( Q_4(x) Q_5(y) + Q_5(x) Q_4(y) \big) + (360 h - 103175) Q_5(x) Q_5(y) \\
    Q_3(x y) = & ~ \big( Q_0(x) Q_3(y) + Q_1(x) Q_2(y) + Q_2(x) Q_1(y) + Q_3(x) Q_0(y) \big) \\
               & + 60 \big( Q_1(x) Q_5(y) + Q_2(x) Q_4(y) + Q_3(x) Q_3(y) + Q_4(x) Q_2(y) \\
               & + Q_5(x) Q_1(y) \big) + 545 \big( Q_2(x) Q_5(y) + Q_3(x) Q_4(y) + Q_4(x) Q_3(y) \\
               & + Q_5(x) Q_2(y) \big) + (h + 3350) \big( Q_3(x) Q_5(y) + Q_4(x) Q_4(y) \\
               & + Q_5(x) Q_3(y) \big) + (10 h + 18020) \big( Q_4(x) Q_5(y) + Q_5(x) Q_4(y) \big) + (65 h \\
               & + 92350) Q_5(x) Q_5(y) 
   \end{split}
  \end{equation*}
  \begin{equation*}
   \begin{split}
    Q_4(x y) = & ~ \big( Q_0(x) Q_4(y) + Q_1(x) Q_3(y) + Q_2(x) Q_2(y) + Q_3(x) Q_1(y) \\
               & + Q_4(x) Q_0(y) \big) - 35 \big( Q_1(x) Q_5(y) + Q_2(x) Q_4(y) + Q_3(x) Q_3(y) \\
               & + Q_4(x) Q_2(y) + Q_5(x) Q_1(y) \big) - 290 \big( Q_2(x) Q_5(y) + Q_3(x) Q_4(y) \\
               & + Q_4(x) Q_3(y) + Q_5(x) Q_2(y) \big) - 1730 \big( Q_3(x) Q_5(y) + Q_4(x) Q_4(y) \\
               & + Q_5(x) Q_3(y) \big) + (h - 9250) \big( Q_4(x) Q_5(y) + Q_5(x) Q_4(y) \big) + \big( 10 h \\
               & - 47430 \big) Q_5(x) Q_5(y) \\
    Q_5(x y) = & ~ \big( Q_0(x) Q_5(y) + Q_1(x) Q_4(y) + Q_2(x) Q_3(y) + Q_3(x) Q_2(y) \\
               & + Q_4(x) Q_1(y) + Q_5(x) Q_0(y) \big) + 10 \big( Q_1(x) Q_5(y) + Q_2(x) Q_4(y) \\
               & + Q_3(x) Q_3(y) + Q_4(x) Q_2(y) + Q_5(x) Q_1(y) \big) + 65 \big( Q_2(x) Q_5(y) \\
               & + Q_3(x) Q_4(y) + Q_4(x) Q_3(y) + Q_5(x) Q_2(y) \big) + 360 \big( Q_3(x) Q_5(y) \\
               & + Q_4(x) Q_4(y) + Q_5(x) Q_3(y) \big) + 1870 \big( Q_4(x) Q_5(y) + Q_5(x) Q_4(y) \big) \\
               & + (h + 9450) Q_5(x) Q_5(y) \qquad\qquad\qquad\qquad\qquad\qquad\qquad\qquad\quad~~~
   \end{split}
  \end{equation*}

  \item Subject to the above commutation and Adem relations, $\G$ becomes a free 
  left module over $E_0$.  A basis consists of monomials in the generators $Q_i$ 
  which are in the form $Q_0^m Q_{i_1} \cdots Q_{i_n}$, with $m \geq 0$, 
  $n \geq 0$ ($n = 0$ corresponds to $Q_0^m$), and $1 \leq i_k \leq 5$.  The 
  grading on $\G$ refers to the sum of the exponents in each of these monomials.  
 \end{itemize}
\end{ex}

Let $E$ be a Morava $E$-theory of height $n$ at the prime $p$.  Given any 
$K(n)$-local commutative $E$-algebra $A$, the elements $Q_i$ above are examples 
of additive power operations having $A^0$ as both domain and range.  Such 
operations arise from the total power operation 
\begin{equation}
 \label{totalpo}
 \psi^p \co A^0 \to A^0(B\Sigma_p) / J \cong A^0[\A] / \big( w(\A) \big) 
\end{equation}
where $J$ is a transfer ideal, and $w(\A) \in E^0[\A]$ is a monic polynomial of 
degree $r \ce 1 + p + p^2 + \cdots + p^{n - 1}$ (cf.~\eqref{psi}).  We define 
{\em individual power operations} $Q_i \co A^0 \to A^0$ by the formula 
\begin{equation}
 \label{Q}
 \psi^p(x) = \sum_{i = 0}^{r - 1} Q_i(x) \, \A^i 
\end{equation}
In particular, given a space $X$, each $Q_i$ acts on $E^0 X$ if we take $A$ to 
be the spectrum of functions from $\Sigma_+^\infty X$ to $E$.  These individual 
power operations $Q_i$ generate the \DL algebra of the $E$-theory 
(cf.~\cite[proof of Theorem 3.10]{p3}).  

\begin{prop}
 \label{prop:Q}
 Let $E$ be a Morava $E$-theory of height $n = 2$ at the prime $p$.  Given any 
 $K(2)$-local commutative $E$-algebra $A$, let $\psi^p$ be the total power 
 operation in \eqref{totalpo} with the polynomial 
 \begin{equation}
  \label{wi}
  w(\A) = \A^{p + 1} + w_p \A^p + \cdots + w_1 \A + w_0 
 \end{equation}
 and let $Q_i$, $0 \leq i \leq p$ be the corresponding individual power 
 operations in \eqref{Q}.  For $\mu = 1$ and $\mu = 2$, let 
 \[
  t_{\mu,\,p} \co A^0 \to p^{-1} A^0 
 \]
 be the topological Hecke operators that are defined as in \eqref{tp} and 
 \eqref{t2p} from the above total power operation $\psi^p$ on $A^0$.  We have 
 the following identities: 
 \[
  t_{1,\,p} = \frac{1}{p} \sum_{i = 0}^p c_i \, Q_i \qquad \ad \qquad t_{2,\,p} 
  = \frac{1}{p^2} \sum_{j = 0}^p \sum_{i = 0}^j w_0^i \, d_{j - i} \, Q_i Q_j 
 \]
 where 
 \begin{equation*}
  \begin{split}
   c_i = & \left\{\!\!
   \begin{array}{ll}
    p + 1 & \qquad\, i = 0 \\
    -\sum_{k = 0}^{i - 1} w_{p + 1 + k - i} \, c_k + (p + 1 - i) \, 
    w_{p + 1 - i} & \qquad 1 \leq i \leq p 
   \end{array}
   \right.\\
   d_\T = & \left\{\!\!
   \begin{array}{ll}
    1 & \qquad\qquad\qquad\qquad\quad\, \T = 0 \\
    -\sum_{k = 0}^{\T - 1} w_0^{\T - k - 1} w_{\T - k} \, d_k & 
    \qquad\qquad\qquad\qquad\quad\, 1 \leq \T \leq p 
   \end{array}
   \right.\qquad\quad~\,
  \end{split}
 \end{equation*}
 In fact, for $i \geq 1$ and $\T \geq 1$, 
 \begin{equation*}
  \begin{split}
    c_i = & ~ i \sum_{\stackrel{\scriptstyle m_1 + 2 m_2 + \cdots + v m_v = i}
            {m_s \geq 1}} (-1)^{m_1 + \cdots + m_v} \, 
            \frac{(m_1 + \cdots + m_v - 1)!}{m_1! \cdots m_v!} \, 
            w_{p + 1 - 1}^{m_1} \cdots w_{p + 1 - v}^{m_v} \\
   d_\T = & ~ \sum_{n = 0}^{\T - 1} (-1)^{\T - n} \, w_0^n 
            \sum_{\stackrel{\scriptstyle m_1 + \cdots + m_{\T - n} = \T}
            {m_s \geq 1}} w_{m_1} \cdots w_{m_{\T - n}} 
  \end{split}
 \end{equation*}
\end{prop}

\begin{proof}
 By definitions \eqref{tp} and \eqref{Q}, 
 \begin{equation*}
  \begin{split}
   t_{1,\,p}(x) = & ~ \frac{1}{p} \sum_{j = 0}^p \psi^p_j(x) \\
                = & ~ \frac{1}{p} \sum_{j = 0}^p \sum_{i = 0}^p Q_i(x) \A_j^i \\
                = & ~ \frac{1}{p} \sum_{i = 0}^p \left( \sum_{j = 0}^p 
                    \A_j^i \right) Q_i(x) 
  \end{split}
 \end{equation*}
 Since the parameters $\A_j$ are the roots of $w(\A)$ in \eqref{wi}, the 
 formulas for $c_i = \sum_{j=0}^p \A_j^i$ then follow from Newton's and Girard's 
 formulas relating power sums and elementary symmetric functions (see, e.g., 
 \cite[Problem 16-A]{cc} and Section \ref{subsec:correction}).  

 For $t_{2,\,p}(x)$, we first write by \eqref{t2p} that 
 \begin{equation*}
  \begin{split}
   t_{2,\,p}(x) = & ~ \frac{1}{p^2} \psi^p \big( \psi^p(x) \big) \\
                = & ~ \frac{1}{p^2} \psi^p \sum_{j = 0}^p Q_j(x) \A^{\,j} \\
                = & ~ \frac{1}{p^2} \sum_{j = 0}^p \psi^p \big( Q_j(x) \big) 
                  \psi^p(\A)^{\,j} \\
                = & ~ \frac{1}{p^2} \sum_{j = 0}^p \left( \sum_{i = 0}^p 
                  Q_i Q_j(x) \A^i \right) \widetilde{\A}^{\,j} 
  \end{split}
 \end{equation*}
 Since the target of $t_{2,\,p}$ is $p^{-1} A^0$, the above identity should 
 simplify to contain neither $\A$ nor $\widetilde{\A}$.  Thus in view of 
 \eqref{w0}, we rewrite 
 \[
  t_{2,\,p}(x) = \frac{1}{p^2} \sum_{j = 0}^p \sum_{i = 0}^j w_0^i \, Q_i Q_j(x) 
  \, \widetilde{\A}^{\,j - i} 
 \]
 For $0 \leq \T \leq p$, we next express each $\widetilde{\A}^\T$ as a 
 polynomial in $\A$ with coefficients in $E^0$, and verify that the constant 
 term of this polynomial is $d_\T$ as stated.  

 The case $\T = 0$ is clear.  For $1 \leq \T \leq p$, we have 
 \begin{equation*}
  \begin{split}
   \widetilde{\A}^\T = & \left( \frac{w_0}{\A} \right)^{\!\T} 
                         \qquad\qquad\qquad\qquad\qquad\qquad\qquad\qquad 
                         \text{by \eqref{w0}} \\
                     = & ~ \frac{w_0^{\T - 1} (-\A^{p + 1} - w_p \A^p - \cdots 
                         - w_1 \A)}{\A^\T} 
                         \qquad\qquad~\, \text{by \eqref{wi}} \\
                     = & ~ w_0^{\T - 1} (-\A^{p + 1 - \T} - w_p \A^{p - \T} 
                         - \cdots - w_{\T + 1} \A) \\
                       & - w_0^{\T - 1} w_\T 
                         - \frac{w_0^{\T - 1} w_{\T - 1}}{\A} - \cdots 
                         - \frac{w_0^{\T - 1} w_1}{\A^{\T - 1}} \\
                     = & ~ w_0^{\T - 1} (-\A^{p + 1 - \T} - w_p \A^{p - \T} 
                         - \cdots - w_{\T + 1} \A) \\
                       & - w_0^{\T - 1} w_\T 
                         - w_0^{\T - 2} w_{\T - 1} \widetilde{\A} - \cdots 
                         - w_1 \widetilde{\A}^{\T - 1} 
                         \quad~~~ \text{by \eqref{w0}} 
  \end{split}
 \end{equation*}
 and thus 
 \begin{equation*}
  \begin{split}
   d_\T = & -w_0^{\T - 1} w_\T - w_0^{\T - 2} w_{\T - 1} d_1 - \cdots 
            - w_1 d_{\T - 1} \qquad\qquad\quad~ \\
        = & -\sum_{k = 0}^{\T - 1} w_0^{\T - k - 1} w_{\T - k} \, d_k 
  \end{split}
 \end{equation*}
 This gives the first identity for $d_\T$ stated in the proposition.  From this 
 relation, we show the second stated identity for $d_\T$ by induction on $\T$.  
 The base case $\T = 1$, with $d_1 = -w_1$, can be checked directly.  For 
 $\T \geq 2$, we first rewrite 
 \[
  d_\T = -w_0^{\T - 1} w_\T - \sum_{r = 1}^{\T - 1} 
  w_0^{r - 1} w_r \, d_{\T - r} 
 \]
 Expanding $d_{\T - r}$ for $1 \leq r \leq \T - 1$ by the induction hypothesis, 
 we then arrange $\sum_{r = 1}^{\T - 1} w_0^{r - 1} w_r \, d_{\T - r}$ above 
 into a sum indexed by $n$ with $0 \leq n \leq \T - 2$, where we collect terms 
 that contain exactly an $n$-power of $w_0$: 
 \begin{equation*}
  \begin{split}
   d_\T = & -w_0^{\T - 1} w_\T - \sum_{n = 0}^{\T - 2} 
            \Bigg( \sum_{r = 1}^{n + 1} w_0^{r - 1} w_r \, 
            (-1)^{\T - r - (n - r + 1)} w_0^{n - r + 1} \\
          & \qquad\qquad\qquad\qquad\, \sum_{\stackrel{\scriptstyle m_{1,r} 
            + \cdots + m_{\T - r - (n - r + 1),r} = \T - r}{m_{s,r} \geq 1}} 
            w_{m_{1,r}} \cdots w_{m_{\T - r - (n - r + 1),r}} \Bigg) \\\\
        = & -w_0^{\T - 1} w_\T - \sum_{n = 0}^{\T - 2} (-1)^{\T - n - 1} w_0^n 
            \sum_{r = 1}^{n + 1} w_r \sum_{\stackrel{\scriptstyle m_{1,r} 
            + \cdots + m_{\T - n - 1,r} = \T - r}{m_{s,r} \geq 1}} w_{m_{1,r}} 
            \cdots w_{m_{\T - n - 1,r}} \\
        = & -w_0^{\T - 1} w_\T + \sum_{n = 0}^{\T - 2} (-1)^{\T - n} w_0^n 
            \sum_{\stackrel{\scriptstyle m_1 + \cdots + m_{\T - n} = \T}{m_s 
            \geq 1}} w_{m_1} \cdots w_{m_{\T - n}} \\
        = & ~ \sum_{n = 0}^{\T - 1} (-1)^{\T - n} w_0^n 
            \sum_{\stackrel{\scriptstyle m_1 + \cdots + m_{\T - n} = \T}{m_s 
            \geq 1}} w_{m_1} \cdots w_{m_{\T - n}} 
  \end{split}
 \end{equation*}
\end{proof}

\begin{rmk}
 \label{rmk:Adem}
 In the proof above, the method for computing $t_{2,\,p}$ applies more 
 generally.  It enables us to find formulas for Adem relations---one for each 
 $Q_i Q_0$, $1 \leq i \leq p$---in terms of the coefficients $w_j$ of $w(\A)$ 
 (see Example \ref{ex:gamma} and cf.~\cite[proof of Proposition 3.6\,(iv)]{p3}).  
 We can also express Cartan formulas using these coefficients.  However, 
 commutation relations are determined by both the terms $w_j$ and $\psi^p(w_j)$ 
 (cf.~Remark \ref{rmk:invl}).  
\end{rmk}

\begin{ex}
 \label{ex:t5}
 For $p = 5$, by Proposition \ref{prop:Q}, we compute from \eqref{w} that 
 \begin{equation*}
  \begin{split}
   t_{1,\,5} = & ~ \frac{6}{5} Q_0 + 2 Q_1 + 6 Q_2 + 26 Q_3 + 126 Q_4 
                 + (h + 600) Q_5 \\
   t_{2,\,5} = & ~ \frac{1}{25} \big( Q_0 Q_0 + h Q_0 Q_1 + (h^2 - 275) Q_0 Q_2 
                 + (h^3 - 550 h + 1500) Q_0 Q_3 \\
               & + (h^4 - 825 h^2 + 3000 h + 71250) Q_0 Q_4 + (h^5 - 1100 h^3 
                 + 4500 h^2 
  \end{split}
 \end{equation*}
 \begin{equation*}
  \begin{split}
  \qquad\qquad & + 218125 h - 818750) Q_0 Q_5 + 5 Q_1 Q_1 + 5 h Q_1 Q_2 + (5 h^2 
                 - 1375) Q_1 Q_3 \\
               & + (5 h^3 - 2750 h + 7500) Q_1 Q_4 + (5 h^4 - 4125 h^2 
                 + 15000 h \\
               & + 356250) Q_1 Q_5 + 25 Q_2 Q_2 + 25 h Q_2 Q_3 + (25 h^2 
                 - 6875) Q_2 Q_4 + (25 h^3 \\
               & - 13750 h + 37500) Q_2 Q_5 + 125 Q_3 Q_3 + 125 h Q_3 Q_4 
                 + (125 h^2 \\
               & - 34375) Q_3 Q_5 + 625 Q_4 Q_4 + 625 h Q_4 Q_5 
                 + 3125 Q_5 Q_5 \big) 
  \end{split}
 \end{equation*}
 Observe that $t_{1,\,5} \equiv \sum_{i = 0}^5 \si_{i - 1}(5) \, Q_i \md \d$, 
 where $\d = h - 26$ as in Example \ref{ex:mfe0}.  
\end{ex}

\begin{thm}
 \label{thm:center}
 Let $E$ be a Morava $E$-theory of height $2$ at the prime $p$, and let $\G$ be 
 its \DL algebra.  Define $\tilde{t}_{\mu,\,p} \ce p^\mu t_{\mu,\,p}$ for 
 $\mu = 1$ and $\mu = 2$, where $t_{\mu,\,p}$ are the topological Hecke 
 operators in Proposition \ref{prop:Q}.  Then $\tilde{t}_{2,\,p}$ lies in the 
 center of $\G$ but $\tilde{t}_{1,\,p}$ does not.  
\end{thm}

\begin{proof}
 By Proposition \ref{prop:Q}, both $\tilde{t}_{1,\,p}$ and $\tilde{t}_{2,\,p}$ 
 are contained in $\G$.  In fact, $\tilde{t}_{2,\,p} = \psi^p \circ \psi^p$, and 
 is thus a ring homomorphism.  Note that $\tilde{t}_{2,\,p}(\A) = \A$ as a 
 result of the relation \eqref{w0}.  Since 
 $\psi^p = Q_0 + Q_1 \A + \cdots + Q_p \A^p$, we can then write the three-fold 
 composite $\psi^p \circ \psi^p \circ \psi^p$ in two ways: 
 \begin{equation*}
  \begin{split}
   \tilde{t}_{2,\,p} (Q_0 + Q_1 \A + \cdots + Q_p \A^p) = 
    & ~ \tilde{t}_{2,\,p} Q_0 + (\tilde{t}_{2,\,p} Q_1) (\tilde{t}_{2,\,p} \A) 
      + \cdots + (\tilde{t}_{2,\,p} Q_p) (\tilde{t}_{2,\,p} \A)^p \\
                                                        = 
    & ~ \tilde{t}_{2,\,p} Q_0 + (\tilde{t}_{2,\,p} Q_1) \A + \cdots 
      + (\tilde{t}_{2,\,p} Q_p) \A^p \\
   (Q_0 + Q_1 \A + \cdots + Q_p \A^p) \tilde{t}_{2,\,p} = 
    & ~ Q_0 \tilde{t}_{2,\,p} + (Q_1 \tilde{t}_{2,\,p}) \A + \cdots 
      + (Q_p \tilde{t}_{2,\,p}) \A^p 
  \end{split}
 \end{equation*}
 For each $0 \leq i \leq p$, comparing the coefficients for $\A^i$, we see that 
 $\tilde{t}_{2,\,p}$ commutes with $Q_i$.  Thus $\tilde{t}_{2,\,p}$ lies in the 
 center of $\G$.  

 It remains to show that the other operation $\tilde{t}_{1,\,p}$ is not central 
 in $\G$.  Analogous to the proof of Proposition \ref{prop:Q}, we find that the 
 term $Q_0 Q_1$ has coefficient $w_2$ in the Adem relation for $Q_1 Q_0$ (see 
 Remark \ref{rmk:Adem}).  Since $p | w_i$ for $2 \leq i \leq p$, we then compute 
 by Proposition \ref{prop:Q} that 
 \begin{equation*}
  \begin{split}
   Q_1 \tilde{t}_{1,\,p} = & ~ Q_1 \big( (p + 1) Q_0 - w_p Q_1 + c_2 Q_2 
                             + \cdots + c_p Q_p \big) \\
                    \equiv & ~ Q_1 Q_0 + Q_1 (c_2 Q_2 + \cdots + c_p Q_p) \md p 
  \end{split}
 \end{equation*}
 Thus the reduction of $Q_1 \tilde{t}_{1,\,p}$ modulo $p$ does not contain 
 $Q_0 Q_1$.  On the other hand, the term $Q_0 Q_1$ in 
 \[
  \tilde{t}_{1,\,p} Q_1 = \big( (p + 1) Q_0 - w_p Q_1 + c_2 Q_2 + \cdots 
  + c_p Q_p \big) Q_1 
 \]
 has coefficient $p + 1$.  Therefore 
 $\tilde{t}_{1,\,p} Q_1 \neq Q_1 \tilde{t}_{1,\,p}$.  
\end{proof}


\newpage
\renewcommand\refname{}
\newcommand{\AX}[1]{\href{http://arxiv.org/abs/#1}{arXiv:#1}}
\newcommand{\MRn}[2]{\href{http://www.ams.org/mathscinet-getitem?mr=#1}{MR#1#2}}
\newcommand{\name}{TateNormalLevelResolutions.pdf}
\wt{.}

\end{document}